\documentclass[envcountsame, runningheads]{llncs}
\usepackage{amssymb}
\usepackage{amsmath}
\usepackage{mathtools}
\usepackage{relsize}

\DeclarePairedDelimiter{\floor}{\lfloor}{\rfloor}
\usepackage{framed}

\usepackage{url}
\usepackage{algpseudocode}
\usepackage{algorithm}

\usepackage{enumitem}
\usepackage{xcolor}
\usepackage{comment}
\usepackage{tikz}
\usepackage{epsfig}
\usetikzlibrary{shapes,backgrounds, patterns}

\colorlet{bscolor}{blue}
\colorlet{skcolor}{orange}

\newcommand{\Omit}[1]{}

\newcommand{\cf}{CFON}
\newcommand{\chicf}{\chi_{ON}(G)}
\newcommand{\chion}{\chi_{ON}(G)}
\newcommand{\chicfp}{\chi^*_{ON}(G)}
\newcommand{\partialcf}{CFON*}
\newcommand{\cfcn}{CFCN}
\newcommand{\chicn}{\chi_{CN}(G)}

\usepackage{array}

\usepackage{xcolor}
\usepackage{comment}
\usepackage{tikz}
\usepackage{epsfig}
\usepackage{tcolorbox}
\tcbuselibrary{breakable}
\newtcolorbox{mybox}[2][]{colbacktitle=white,colback=white,coltitle=black,title={#2},fonttitle=\bfseries,#1, left = 2mm, right = 2mm, breakable}

\begin{document}
\title{Combinatorial Bounds for Conflict-free Coloring on Open Neighborhoods}
\author{
Sriram Bhyravarapu
\and
Subrahmanyam Kalyanasundaram
}
\institute{Department of Computer Science and Engineering, IIT Hyderabad \\
\email{\{cs16resch11001, subruk\}@iith.ac.in}
}

\maketitle

\begin{abstract}
In an undirected graph $G$, a conflict-free coloring 
with respect to open neighborhoods (denoted by \cf{} coloring)
is an assignment of colors to the vertices such that 
every vertex has a uniquely colored vertex in its open 
neighborhood.  
The minimum number of colors required for a \cf{} coloring of $G$ 
is the \cf{} chromatic number of $G$, denoted by $\chicf$.

The decision problem that asks whether $\chicf \leq k$ is NP-complete.
Structural as well as algorithmic aspects of this problem have been 
well studied. 
We 
obtain
the following results for $\chicf$:
\begin{itemize}
\item Bodlaender, Kolay and Pieterse [WADS 2019] showed the upper bound
$\chicf \leq {\sf fvs}(G) + 3$, where 
${\sf fvs}(G)$ denotes the size of a minimum feedback vertex set of $G$. 
We show the improved bound of 
$\chicf \leq {\sf fvs}(G) + 2$, which is tight, thereby answering an 
open question in the above paper.

    \item We study the relation between $\chicf$ and the pathwidth of the 
    graph $G$, denoted ${\sf pw}(G)$. The above paper from 
    WADS 2019 showed the upper bound
$\chicf \leq 2{\sf tw}(G) + 1$ where ${\sf tw}(G)$ stands for the treewidth of $G$. This implies
an upper bound of $\chicf \leq 2{\sf pw}(G) + 1$.
We show an improved bound of $\chicf \leq \floor{\frac{5}{3}({\sf pw}(G) + 1)}$. 

    \item We prove new bounds for $\chicf$ with respect 
    to the structural parameters neighborhood diversity 
    and distance to cluster, improving the existing results of Gargano and Rescigno [Theor. Comput. Sci. 2015] and Reddy 
    [Theor. Comput. Sci. 2018], respectively. Furthermore, 
    our techniques also yield improved bounds 
    for the closed neighborhood variant of the problem.
    
    \item We also study the partial coloring variant of the \cf{} coloring problem, which allows vertices to
    be left uncolored. Let $\chicfp$ denote the minimum number
    of colors required to color $G$ as per this 
     variant. Abel et. al. [SIDMA 2018]
    showed that $\chicfp \leq 8$ when $G$ is planar. They asked if fewer colors
    would suffice for planar graphs. We answer this question  by showing that $\chicfp \leq 5$ for all planar $G$. This approach also yields the bound $\chicfp \leq 4$ for all outerplanar $G$.
\end{itemize}
All our bounds are a result of constructive algorithmic procedures.
\end{abstract}

\section{Introduction}
A \emph{proper coloring} of a graph is an assignment of a color to every vertex of the graph such that adjacent vertices receive
 distinct colors.
Conflict-free coloring is a variant of the graph coloring problem. 
A conflict-free coloring of a graph $G$ 
is a coloring such that for every vertex in $G$, there exists a uniquely colored vertex in its neighborhood. 
This problem was first introduced in 2002 by Even, Lotker, Ron and Smorodinsky \cite{Even2002}. 
This problem was originally motivated by wireless communication systems, 
where the base stations and clients
have to communicate with each other.
Each base station is assigned a frequency and if two base stations with 
the same frequency communicate with the same client, it leads to interference. 
So for each client, it is ideal to 
have a base station with a unique frequency. 
Since each frequency band is expensive, there is a need 
to minimize the number of frequencies used 
by the base stations.

Over the past two decades, this problem has been very well studied, see for instance
the survey by Smorodinsky \cite{smorosurvey}.
The conflict-free coloring problem has been studied with respect to the open  neighborhood and the closed neighborhood.
In this paper, we focus on the 
open neighborhood variant of the problem. 

\begin{definition}[Conflict-Free Coloring]\label{def:cf}
A \emph{\cf{} coloring} of a graph $G = (V,E)$ using $k$ colors is an 
assignment $C:V(G) \rightarrow \{1, 2, \ldots, k\}$ such that for every $v \in V(G)$, 
there exists an $i \in \{ 1, 2, \ldots, k\}$ such that  $|N(v) \cap C^{-1}(i)| = 1$.
The smallest number of colors required for a \cf{} coloring
of $G$ is called the \emph{\cf{} chromatic number of $G$}, denoted by $\chicf$.

The closed neighborhood variant of the problem, \emph{\cfcn{} coloring}, is obtained by replacing the open neighborhood $N(v)$ by the closed
neighborhood $N[v]$ in the above. The corresponding chromatic number is
denoted by $\chicn$.
\end{definition}
The \cf{} coloring problem and many of its variants are known to be 
NP-complete~\cite{planar,gargano2015}. It was further shown in \cite{gargano2015} that the \cf{} coloring problem is hard 
to approximate within a factor of $n^{1/2 - \varepsilon}$,
unless P = NP. 
Since the problem is NP-hard, the parameterized aspects of the problem 
have been studied. 
The problems are fixed parameter tractable when 
parameterized by vertex cover number, neighborhood 
diversity~\cite{gargano2015}, distance to cluster
\cite{vinod2017}, and more recently, treewidth~\cite{Boe2019,Sak2019}. 
This problem has attracted special interest for 
graphs arising out of intersection of 
geometric objects, see for instance,
\cite{Keller2018,sandor2017,chen2005}. 

The \cf{} coloring problem
is considered as the harder of the open and closed neighborhood variants, see
for instance, remarks in \cite{Keller2018,Pach2009}. It is easy to construct example
graphs $G$, for which $\chicn = 2$ and $\chi_{ON}(G) = \Theta(\sqrt{n})$.
Pach and Tardos \cite{Pach2009} showed that for any graph $G$ on 
$n$ vertices,
the closed neighborhood chromatic number $\chicn = O(\log^2 n)$.
The corresponding best bound \cite{Pach2009,pcheilaris2009} for open neighborhood is 
$\chion = O(\sqrt n)$.


Another variant that has been studied \cite{planar} is the partial coloring variant:
\begin{definition}[Partial Conflict-Free Coloring]
A \emph{partial conflict-free coloring on open neighborhood}, denoted by \partialcf{}, of a graph $G= (V,E)$ using $k$ colors is an 
assignment $C:V(G) \rightarrow \{1, 2, \ldots, k, \mbox{unassigned}\}$ such that for every $v \in V(G)$, 
there exists an $i \in \{1, 2, \ldots, k\}$ such that  $|N(v) \cap C^{-1}(i)| = 1$.

The corresponding \partialcf{} chromatic number is denoted $\chicfp$.
\end{definition}
The key difference between \partialcf{} coloring and \cf{} coloring is that in the partial variant, we allow some
vertices to be not assigned a color. 
If a graph can be \partialcf{} colored using $k$ colors, then all the uncolored vertices can be assigned 
the color $k+1$, and thus is a 
\cf{} coloring using $k+1$ colors.

\subsection{Our Results and Discussion}
In this paper, we obtain improved bounds for $\chicf$ under different
settings. More importantly, all our bounds are a result of
constructive algorithmic procedures and hence can easily
be converted into respective algorithms. 
We summarize our results below:

\vspace{-0.1in}
\begin{enumerate}
        \item In Section \ref{sec:pw}, we show that $\chicf \leq \floor{\frac{5}{3}({\sf pw} (G)+1)}$ where ${\sf pw}(G)$ denotes the pathwidth of $G$. The previously best known bound in terms of ${\sf pw}(G)$ was
    $\chicf \leq 2{\sf pw}(G) + 1$, implied by the results in \cite{Boe2019}.
    
    To the best of our knowledge, this is the first upper bound for $\chion$
    in terms of pathwidth, which does not follow from treewidth. Our 
    bound follows from an algorithmic procedure and uses an intricate analysis.
    We are unable to 
     generalize our bound in terms of treewidth because we 
    crucially use a fact (stated as Theorem \ref{thm:pwneighbor}) that applies to path decomposition,
    but does not seem to apply to tree decomposition. 
    It will be of interest to see if this hurdle can be overcome to obtain
    an equivalent bound in terms of treewidth.
    
    There are graphs $G$ for which $\chion = {\sf tw}(G) + 1 = {\sf pw}(G)$. 
    It would be interesting to close the gaps between the respective 
    upper and lower bounds. 
 
 
    \item In Section \ref{sec:fvs}, we show that $\chicf \leq {\sf fvs}(G) + 2$, where ${\sf fvs}(G)$ 
    denotes the size of a minimum feedback vertex set of $G$.
    This bound is tight and is an improvement over the bound $\chicf \leq {\sf fvs}(G) + 3$ by Bodlaender, Kolay and Pieterse \cite{Boe2019}. 
    
    \item In Section \ref{sec:nd}, we give improved bounds with respect to neighborhood diversity parameter. 
    Gargano and Rescigno \cite{gargano2015} showed that 
    $\chion \leq \chi_{ON}(H) + cl(G) + 1$ and  $\chicn \leq \chi_{CN}(H) + ind(G) + 1$. Here $H$ is the type graph of $G$, 
    while $cl(G)$ and $ind(G)$ denote the number of cliques and independent
    sets respectively in the type partition of $G$. We present 
    the improvements $\chion \leq \chi_{ON}(H) + cl(G)/2 + 2$ and  $\chicn \leq \chi_{CN}(H) + ind(G)/3 + 3$. 
    
    \item In Section \ref{sec:dc}, we show that $\chion \leq {\sf dc}(G)+3$, where ${\sf dc}(G)$ is the distance to cluster parameter of $G$.
    This
    is an improvement over the previous bound \cite{vinod2017} of $2{\sf dc}(G)+3$.
    Our bound is nearly tight since there are graphs for which $\chicf ={\sf dc}(G)$.
    Using a similar approach, we obtain the improved bound
    $\chicn \leq \max\{3, {\sf dc}(G)+1\}$.
    
    For the results in terms of parameters neighborhood diversity and distance to cluster,  the obvious open questions are to improve the bounds and/or to provide tight examples. 
    \item When $G$ is planar, we show that $\chicfp \leq 5$. This improves
    the previous best known bound by Abel et al. \cite{planar} of $\chicfp \leq 8$. 
    The same approach helps us show that $\chicfp \leq 4$, when $G$
    is an outerplanar graph. 
    These two results are
    discussed in Section \ref{sec:planar}.

    There are planar graphs $G$ for which $\chi^*_{ON}(G) = 4$, which shows that our bound is nearly tight and leaves a gap of 1
between the upper and lower bounds. It will be of interest to close this gap. 
    \item For outerplanar graphs $G$, the bound $\chicfp \leq 4$  
    implies a bound of $\chicf \leq 5$.
    We show a better bound of $\chicf \leq 4$.
\end{enumerate}

\section{Preliminaries}\label{sec:prelim}
In this paper, we consider only simple, finite, undirected and connected graphs. If the graph is not 
connected, we color each of the components independently. 
Also, we assume that the graphs do not have isolated vertices as they cannot be \cf{} colored. 
The graph induced by a set of vertices $V'$ in $G$ is denoted $G[V']$. 
For any two vertices $u,v \in V(G)$, the shortest distance between them 
is denoted  $dist(u,v)$. 
The open neighborhood of $v$, denoted $N(v)$, is the set of vertices adjacent to $v$. The closed neighborhood of $v$, denoted $N[v]$, is
$N[v] = N(v) \cup \{v\}$.
The degree of a vertex  $v$ in the graph is denoted $\mbox{deg}(v)$.
The distance, degree and neighborhood restricted to a subgraph $H$ is denoted $dist_H(u,v)$, $\deg_H(v)$ and $N_H(v)$ respectively.

We denote the set $\{1, 2, \dots , q\}$ by $[q]$. 
Throughout this paper, we use the coloring functions $C:V\rightarrow [q]$ and $U:V\rightarrow [q]$ to denote the color assigned to a vertex and a unique color in its neighborhood, respectively. 
For a vertex $v \in V(G)$, if there exists a vertex $w\in N(v)$ such that 
$\{ x \in N(v) \setminus \{w\}$: $C(x) = C(w) \} = \emptyset$, then $w$ is called
a uniquely colored neighbor of $v$.

For theorems marked $\star$, we provide the full proofs in the appendix due to space
constraints.


\section{Pathwidth}\label{sec:pw}

\begin{theorem}[Main Pathwidth Result]\label{thm:mainpw}
Let $G$ be a graph and let ${\sf pw}(G)$ denote the pathwidth of $G$.
Then there exists a \cf{} coloring of $G$ using at most $\floor{\frac{5}{3}({\sf pw}(G) + 1)}$
colors.
\end{theorem}


The proof of this theorem will be a constructive procedure that assigns colors
to the vertices of $G$ from a set of size $5({\sf pw}(G)+1)/3$. We first formally define 
pathwidth. 

\begin{definition}[Path decomposition \cite{parabook}]
A \emph{path decomposition} of a graph $G$ is a sequence $\mathcal P = (X_1,
X_2, \ldots, X_s)$ of bags such that, for every $p \in \{1, 2, \ldots, s\}$, we have 
$X_p \subseteq V(G)$ and the following hold:
\begin{itemize}
    \item For each vertex $v \in V(G)$, there is a $p \in \{1, 2, \ldots, s\}$ such that $v \in X_p$.
    \item For each edge $\{u, v\} \in E(G)$, there is a $p \in \{1, 2, \ldots, s\}$ such that $u, v \in X_p$.
    \item If $v \in X_{p_1}$ and $v \in X_{p_2}$ for some $p_1 \leq p_2$, then $v \in X_p$ for all 
    $p_1 \leq p \leq p_2$.
\end{itemize}
\end{definition}
The \emph{width} of a path decomposition $(X_1,
X_2, \ldots, X_s)$ is $\max_{1 \leq p \leq s}\{|X_p| - 1\}$. The \emph{pathwidth} 
of a graph $G$, denoted ${\sf pw}(G)$, 
is the minimum width over all path decompositions of $G$. 
For the purposes of our algorithm, we need the path decomposition 
to satisfy certain additional properties too. 

\begin{definition}[Semi-Nice Path Decomposition]\label{def:seminice}
A path decomposition  $\mathcal P = (X_1,
X_2, \ldots, X_s)$ is called a \emph{semi-nice path decomposition} if $X_1 = X_s = \emptyset$ and for all $p \in \{2, \ldots, s\}$, exactly one of the following 
hold:
\vspace{-0.05in}
\begin{description}
    \item[SN1.] There is a vertex $v$ such that
     $v \notin X_{p-1}$ and $X_{p} = X_{p-1} \cup \{v\}$. In this case, we say 
     that $X_{p}$ \emph{introduces} $v$. Further, when $X_{p}$ introduces $v$,
     $N(v) \cap X_{p} \neq \emptyset$.
    \item[SN2.] There is a vertex $v$ such that $v \in X_{p-1}$ and 
    $X_{p} = X_{p-1} \backslash \{v\}$. In this case, we say $X_{p}$ \emph{forgets} $v$.
    \item[SN3.] There is a pair of vertices $v, \widehat v$ such that $v, \widehat v \notin X_{p-1}$ and $X_{p} = X_{p-1} \cup \{v, \widehat v\}$. We call such a bag $X_p$
    a \emph{special bag}
    that introduces $v$ and $\widehat v$. Further, in a special bag $X_{p}$ that 
    introduces $v$ and $\widehat v$, it must be true that $N(v) \cap X_{p} = \{\widehat v\}$ and
    $N(\widehat v) \cap X_{p} = \{v\}$.
\end{description}
\end{definition}

\vspace{-0.05in}
We first note that the every graph without isolated vertices has a semi-nice path
decomposition of width ${\sf pw}(G)$.
\begin{theorem}\label{thm:pwneighbor}
Let $G$ be a graph that has no isolated vertices. Then it has a semi-nice path decomposition 
of width ${\sf pw}(G)$.  
\end{theorem}
The proof of the above theorem
is deferred to Section \ref{sec:pwneighbor}, after the proof of the main
theorem of this section.

\noindent\textbf{Algorithm}
We start with a semi-nice path decomposition $\mathcal P = (X_1,
X_2, \ldots, X_q)$ of width ${\sf pw}(G)$. 
We process each bag in the order
$X_1, X_2, \ldots, X_q$.
As we encounter each bag, we assign to the vertices
in the bag a color $C: V(G) \rightarrow [5({\sf pw}(G)+1)/3]$. We will also identify a unique 
color (from its neighborhood) for each vertex $U: V(G) \rightarrow [5({\sf pw}(G)+1)/3]$. We color
the bags such that the below are satisfied:
\vspace{0.1in}
\begin{mybox}{}
\begin{description}
    \item[Invariant 1.] For any bag $X$, if $v, v' \in X$, then $C(v) \neq C(v')$.
    \item[Invariant 2.] Suppose we have processed bags $X_1$ to 
    $X_p$, where $p \geq 2$. At this point, the induced graph 
    $G[\cup_{1 \leq j \leq p} X_j]$ is \cf{} colored.
    
    \item[Invariant 3.] For every vertex $v$ that appears in the bags processed, $U(v)$
    is set as $C(w)$ for a neighbor $w$ of $v$.
    Once $U(v)$ is assigned, it is ensured that for all ``future'' 
    neighbors $v'$ of $v$, $C(v') \neq U(v)$, thereby ensuring that $U(v)$ is retained 
    as a unique color in $N(v)$.
\end{description}
\vspace{-0.15in}
\end{mybox}
\noindent \textbf{Definitions required for the algorithm:} For each bag $X$, we define the set of \emph{free colors}, as $F(X) = 
 \{U(x): x\in X\} \setminus \{C(x): x\in X\}$. 
That is, $F(X)$ is the set of colors that appear in $X$ as unique
colors of vertices in $X$, but not as colors of any vertex. Further, we partition $F(X)$ into
two sets $F_1(X)$ and $F_{> 1}(X)$. They are defined as 
 $F_1(X) = \{c \in F(X) : 
\left |\{x\in X: U(x) = c \}\right | = 1 \}$ and $F_{>1}(X) = \{c \in F(X) : 
|\{x\in X: U(x) = c \}| > 1 \}$. 
A vertex $v$ that appears in a bag $X$ is called a \emph{needy vertex} (or simply 
\emph{needy}) in $X$, if
$U(v) \in F(X)$. For a bag $X$, we say that a set $S \subseteq X$ is an 
\emph{expensive subset} 
if $|\cup_{w \in S} \{C(w), U(w)\}| = 2|S|.$

When going through the sequence of bags in the semi-nice path decomposition, the bags $X$ that forget 
a vertex only contain vertices that 
have already been assigned colors and  
hence no action needs to be taken. 
When we move from 
a bag $X'$ to the next bag $X$ that introduces 
either one vertex 
or two vertices, we need to handle the introduced vertices.
Let us first consider the bags that introduce one vertex, say $v$.
For all bags that introduce one vertex, we assign $C(v)$ and $U(v)$ as per the below rules.



\begin{mybox}{For bags that introduce one vertex $v$}
\noindent \textbf{Rule 1 for assignment of $C(v)$:} 
\vspace{-0.07in}
    \begin{itemize}
        \item If there exists a color $c \in F_1(X') \setminus 
         \{U(x): {x\in N(v) \cap X'}\}$,
        then we assign $C(v) = c$. If there are more than one such color $c$, choose a $c$ such
        that $|\{x : x\in X', C(U^{-1}(c)) = U(x)\}|$ is minimized. Note that for all $c \in F_1(X')$,
        there is a unique vertex $w \in X'$ such that $U(w) = c$, and hence $U^{-1}(c)$
        is well defined.
    
        \item If 
        $F_1(X') \setminus  \{U(x): x\in N(v) \cap X'\} = \emptyset$,
        we check if there exists a color $c \in F_{>1}(X') \setminus 
         \{U(x): x\in N(v) \cap X'\}$. 
        If so, we assign $C(v) = c$. If there are multiple such $c$, then we choose one arbitrarily.
        
        \item If $F_1(X') \cup F_{>1}(X') \setminus 
         \{U(x): x\in N(v) \cap X'\} = \emptyset$,
        then there are no free colors that can be assigned as $C(v)$. 
    We assign $C(v)$ to be a new color (a color not in $\cup_{x \in X'} \{C(x), U(x)\}$). 
    
    \end{itemize}
\noindent \textbf{Rule 2 for assignment of $U(v)$:} 
    We assign $U(v) = C(y)$, where $y \in X'$ is a neighbor of $v$. 
    Such a $y$ exists by Theorem \ref{thm:pwneighbor}.
    If $v$ has multiple neighbors, we follow the
    below priority order:
    \vspace{-0.07in}
    \begin{itemize}
        \item If $v$ has needy vertices in $X'$ as neighbors, we choose
        $y$ as a needy neighbor such that $|\{x : x \in X', U(y) = U(x) \}|$
        is minimized.
        \item If $v$ does not have needy vertices in $X'$ as neighbors, 
        then we choose $y \in X'$ arbitrarily from the set of neighbors of $v$.
    \end{itemize}
\vspace{-0.1in}
\end{mybox}

Now let us consider the case where the bag $X$ is a  special bag 
that introduces two vertices $v$ and $\widehat v$. We assign $C(v), C(\widehat v), U(v), U(\widehat v)$ as per the 
following:
\begin{mybox}{For special bags that introduce two vertices $v$ and $ \widehat v$}
\noindent \textbf{For assignment of $C(v)$ and $C( \widehat v)$}: We 
select one of $v$ and $\widehat v$ arbitrarily, say $v$, to be colored first. 
We use Rule 1 to assign $C(v)$ and then $C(\widehat v)$, in that order. One point to note is that
during the application of Rule 1 here, the part $\{U(x): {x\in N(v) \cap X'}\}$ will not
feature as neither $v$ nor $\widehat v$ have neighbors in $X'$.

\vspace{0.05in}
\noindent \textbf{For assignment of $U(v)$ and $U(\widehat v)$:} 
    Assign $U(v) = C( \widehat v)$ and $U(\widehat v) = C(v)$.
\end{mybox}

It can easily be checked that the above rules maintain the invariants 1, 2 and 3 stated earlier
and hence the algorithm results in a \cf{} coloring of $G$. What remains is to show that $5({\sf pw}(G)+1)/3$ colors are 
sufficient. We first prove a technical result.

\begin{theorem}[Technical Pathwidth Result]\label{thm:technical}
During the course of the algorithm, let $k$ be the size of the largest expensive
subset out of all the bags in the path decomposition. Then 
there must exist a bag of size at least $3k/2$.
\end{theorem}
\begin{proof}
In the sequence of bags seen by the algorithm, let $X$ be the
first bag that has an expensive subset of size $k$. We show
that $|X| \geq 3k/2$.
Let $S = \{v_1, v_2, \ldots, v_k\} \subseteq X$ be an expensive 
subset  of size $k$. For each $v_i$, let $C(v_i) = 2i - 1$ and 
$U(v_i) = 2i$. 

Let $X'$ be the bag that precedes $X$ in the sequence. 
By the choice of $X$, no $k$-expensive subset is  present in 
$X'$. It follows that $S \not \subseteq X'$.
Hence the bag $X$ must introduce a vertex\footnote{In 
the case where $X$ is a special bag that introduces two vertices, 
at most one of the two introduced vertices can be part of an expensive subset.}
 that belongs to $S$. 
Without loss of generality, let $v_k$ be this vertex introduced in $X$. 
Further wlog, let $v_1, \ldots, v_r$
be the needy vertices (in $X'$) of $S$ for some $1 \leq r \leq k$. If
none of the vertices in $S$ are needy, then we have that
$|X| \geq 2|S| = 2k$
and the theorem holds. So we can assume that $r \geq 1$.

Since the vertices $v_1, \ldots, v_r$ are needy 
in $X'$, we have  $\{2, 4, \ldots, 2r\} \subseteq F(X')$.
The vertices $v_{r+1}, \ldots, v_k$ are not needy because there
exist distinct vertices $Z = \{z_{r+1}, \ldots, z_k\}$
in the bag $X$ such that 
$C(z_i) = U(v_i) = 2i$ for $r+1 \leq i \leq k$. We have
three cases. 
In Cases 1 and 2, $X$ is a bag that introduces one vertex $v_k$.
Case 1 is when none of the colors in $F(X')$ was eligible 
to be assigned as $C(v_k)$. Hence $C(v_k)$ is assigned from outside
the set $\cup_{x \in X'} \{C(x), U(x)\}$.
Case 2 is when there are eligible colors in $F(X')$, and 
$C(v_k)$ is chosen from $F(X') \subseteq \cup_{x \in X'} \{C(x), U(x)\}$. 
Case 3 is when $X$ is a special bag that introduces two vertices.

\vspace{0.1in}
\noindent \textbf{Case 1:} $X$ is a bag that introduces one vertex $v_k$ and $2k - 1 \notin 
 \{U(x): x \in X'\}$. 
    There is no
    vertex $x\in X$ with $U(x) = 2k - 1$. To assign a color
    to $v_k$, the algorithm chose a new color. This means that 
    $F(X') \setminus  \{U(x): x \in N(v_k) \cap X'\} = \emptyset$.
    In particular, for each 
     $1 \leq i \leq r$, there\footnote{The vertex $v'_i$ may or may not be the same
     as $v_i$.}
     exists $v'_i \in N(v_k) \cap X'$
    such that $U(v'_i) = 2i$. Hence the colors $2i$, for $1 \leq i \leq r$
    cannot be assigned as $C(v_k)$.
    By Rule 2, we must set $U(v_k)$ to be the $C(y)$
    where $y$ is a needy neighbor of $v_k$. Hence $C(y) = 2k$.
    
    If $U(y) \notin \{1, 2, 3,\ldots, 2k-2\}$,
    then $(S \cup \{y\}) \setminus \{v_k\}$ is a $k$-expensive subset in $X'$, the predecessor of $X$. This 
    contradicts the choice of $X$.
    Hence we can assume that $U(y) \in \{1, 2, 3,\ldots, 2k-2\}$. Since $y$ is needy in $X'$, $U(y)$
    is not $C(v)$ for any $v \in X'$ and hence 
    we conclude $U(y) \notin \{1, 3, \ldots, 2k - 3\}$.
    Further, colors from
    $\{2(r+1), 2(r+2), \ldots, 2k-2\}$ appear as $C(z)$ for the vertices $z \in Z$. Hence 
    $U(y) \in \{2, 4, 6,\ldots, 2r\}$.
    
    Let $U(y) = 2j$ for some $1 \leq j \leq r$. Notice
    that $U(v_j) = 2j$ as well, giving us 
    $|\{x : x \in X', U(y) = U(x) \}| \geq 2$.
    
    By Rule 2, we chose $U(v_k) = C(y)$, where
    $y$ is the needy neighbor that minimizes    
    $|\{x : x \in X', U(y) = U(x) \}|$. 
    We chose $y$ over other needy neighbors $v'_1, \ldots, v'_r$
    of $v_k$.
    Hence there exist $r$ distinct 
    vertices $Y = \{y_1, \ldots, y_r\}$ in the bag $X$, disjoint from $S$, such 
    that $U(y_i) = U(v'_i) = 2i$ for each $1 \leq i \leq r$. 
    
    Note that the set $Y \cup Z$ must be disjoint
    from $S$, but $Y$ and $Z$ may intersect with each other. 
    Since $|Y| + |Z| = k$, we have $|Y \cup Z| \geq k/2$ and therefore $|X| \geq |S| + |Y \cup Z| \geq 3k/2$.

\vspace{0.1in}
\noindent \textbf{Case 2 (Proof Sketch):} 
$X$ is a bag that introduces one vertex $v_k$ and  $2k - 1 \in  \{U(x): x \in X'\}$. 
    For the sake of brevity and clarity, the full proof of Case 2 is deferred to Appendix \ref{sec:case2pw}. 
    The arguments are similar to the ones used in Case 1, but Case 2 requires
    a lengthier treatment. We give a sketch of the proof here. 
    
    In this case, $C(v_k)$ is chosen from $F(X')$. That is,
    $C(v_k)$ is chosen as $U(w) = 2k-1$, 
    for a vertex $w$ that is needy in $X'$, chosen according to 
    Rule 1. 
    It can be first established that $C(w) \in \{2(r+1), 2(r+2), \ldots, 2k - 2\}$.
    
    
    We know that  $\{2, 4, \ldots, 2r\} \subseteq F(X')$. 
    Let $\ell = |\{2, 4, \ldots, 2r\} \cap 
     \{U(x): x \in N(v_k) \cap X'\}|$.
    The color $2k-1$ was chosen as $C(v_k)$ over the
    other colors in $F(X')$. In particular, it was chosen over
    the $r - \ell$ colors in
    $\{2, 4, \ldots, 2r\} \setminus 
     \{U(x): x \in N(v_k) \cap X'\}$. 
    This is used to show the existence 
    of a set $W \subseteq X$ of size $r - \ell$ that is disjoint from $S$.
    
    Now we study why $U(v_k)$ was assigned as 2k. As per Rule 2,
    $U(v_k)$ was assigned as $C(y)$, where $y$ is a needy neighbor
    of $v_k$. There exists at least $\ell$ needy neighbors of 
    $v_k$. This is used to establish the existence of a set 
    $Y \subseteq X$ of size $\ell$ that is disjoint from $S$ and $W$.
    
    We thus have $|W \cup Y| = r$ and $|Z| = k - r$. The sets
    $W \cup Y$ and $Z$ are both disjoint from $S$, but need not be disjoint from each other.
    Hence $|X| \geq |S| + |W \cup Y \cup Z| \geq 3k/2$.
    
\vspace{0.1in}
\noindent \textbf{Case 3 (Proof Sketch):}  $X$ is a special bag that introduces $v_k$ and $\widehat v_k$. 
As in Case 2, the full proof of
Case 3 is given in Appendix \ref{sec:case3pw}. We give a sketch of the proof here.

If $F(X')=\emptyset$, then none of the $k-1$ vertices in $S \cap X'$ are needy in $X'$.
Hence $|X'|\geq 2(k-1)$. This implies that 
$|X|\geq 2(k-1)+2 =2k$ and we are done.  

Else, $|F(X')|\geq 1$.
Let us first note that since $S$ is an expensive subset, so is $S\cup \{\widehat v_k\} \setminus \{v_k\}$.
Since $|F(X')|\geq 1$, at least one of  $C(v_k)$ or $C(\widehat v_k)$ will be chosen from $F(X')$.
Without loss of generality, let 
$v_k$ be a vertex such that $C(v_k) \in F(X')$.
Let $C(v_k)=U(w)=2k-1$, where $w$  is a needy vertex in $X'$, chosen according to Rule 1. 
The rest of the arguments are very similar to the
arguments in the proof of Case 2. 

We finally establish the existence of a set $W \subseteq X$ such that $W$ is disjoint
from $S$ and $|W| = r$. Hence $|X| \geq |S| + |W \cup Z| \geq 3k/2$.
    \qed

\end{proof}

Now we prove the main theorem of this section.

\begin{proof}[Proof of Theorem \ref{thm:mainpw}]
We apply the algorithm in a nice path decomposition
of $G$, which satisfies the condition in Theorem \ref{thm:pwneighbor}.
As stated before, the correctness follows by the stated invariants, 
and what remains to be shown is the bound on the number of colors 
necessary.

Consider any bag $X$ of the path decomposition. Then we have 
$$|\cup_{w \in X} \{C(w), U(w)\}| = |X| + |\mbox{Extra}(X)|, $$
where $\mbox{Extra}(X)$ denotes the set of colors 
that feature as unique colors, but not as colors of vertices
in $X$. 

We construct a subset $Y$ of $X$ as follows. For each color in $\mbox{Extra}(X)$, we include exactly one vertex $y$ in $Y$ 
such that $U(y)$ is that color. We have $|Y| = |\mbox{Extra}(X)|$.
We also have that $Y$ is an expensive subset of $X$. Since
no bag is of size bigger than ${\sf pw}(G) + 1$, it follows that
 by Theorem \ref{thm:technical} that $|Y| \leq 2({\sf pw}(G)+1)/3$. Since $|Y|$ is an integer, we can say $|Y| \leq \floor{2({\sf pw}(G)+1)/3}$. Hence 
$$|\cup_{w \in X} \{C(w), U(w)\}| \leq |X| + \floor{2({\sf pw}(G)+1)/3}. $$

In the algorithm, we need to add a new color to the bag only when 
a bag $X$ is followed by another bag that introduces a vertex. 
Hence we may require one additional color, which brings 
the maximum number of colors needed to $|X| + \floor{2({\sf pw}(G)+1)/3} + 1 \leq
{\sf pw}(G) + \floor{2({\sf pw}(G)+1)/3} + 1 = \floor{5({\sf pw}(G)+1)/3}$.
\qed
\end{proof}

%

\subsection{Proof of Theorem \ref{thm:pwneighbor}}\label{sec:pwneighbor}
A path decomposition  
$(X_1,
X_2, \ldots, X_s)$ is called a \emph{nice path decomposition} if the following 
hold:
\vspace{-0.05in}
\begin{itemize}
    \item $X_1 = X_s = \emptyset$.
    \item For all $p \in \{2, 3, \ldots, s\}$, there is a vertex $v$ such that
    either $v \notin X_{p-1}$
    and $X_{p} = X_{p-1} \cup \{v\}$, or $v \in X_{p-1}$ and 
    $X_{p} = X_{p-1} \backslash \{v\}$. In the former case, we say $X_{p}$ \emph{introduces} 
    $v$, and in the latter case we say $X_{p}$ \emph{forgets} $v$.
\end{itemize}
\vspace{-0.05in}
It is known \cite{parabook} that every graph $G$ has a nice path decomposition 
of width equal to ${\sf pw}(G)$, and that 
every nice path decomposition has exactly
$2|V(G)| + 1$ bags.

Consider a nice path decomposition of the graph $G$.
If all the vertices have a neighbor in the bag that introduces them, then the nice
path decomposition is itself a semi-nice path decomposition, and we are done.
Otherwise, we explain 
how to convert the given nice path decomposition
into a semi-nice path decomposition. 
We say that a bag $X$ is a \emph{violating bag} if it 
introduces a vertex $v$ and $X \cap N(v) = \emptyset$.
The violating bags do not follow the rules SN1, SN2, or SN3 from Definition \ref{def:seminice}. Instead they follow the below rule SN$1'$.
\begin{description}
    \item[SN$1'$.] There is a vertex $v$ such that
     $v \notin X_{p-1}$ and $X_{p} = X_{p-1} \cup \{v\}$. In this case, we say 
     that $X_{p}$ \emph{introduces} $v$. Further, when $X_{p}$ introduces $v$,
     $N(v) \cap X_{p} = \emptyset$.
\end{description}
We say that a path decomposition is a \emph{$t$-violating semi-nice path 
decomposition} if there are $t$ violating bags and the rest 
of the bags obey one of the rules SN1, SN2 or SN3 from Definition \ref{def:seminice}.

We ``fix'' each violating bag by modifying the path decomposition. The fix involves delaying the introduction
of a vertex until it has a neighbor, and possibly creating  a ``special bag'' that introduces two vertices.
Throughout the fix-up process, the path decomposition in hand will be a $t$-violating semi-nice decomposition, with every step of the fix-up decrementing $t$ by one. We now
explain the fix-up process.

\noindent \textbf{Fix-up Process:} Given a $t$-violating semi-nice path decomposition
$\mathcal P = (X_1, X_2, \ldots, X_s)$, we explain how
to obtain a $(t-1)$-violating semi-nice path decomposition $\mathcal P'$.

Let $X_{p_1}$ be a violating bag that introduces the vertex $v$, which
is forgotten by the bag $X_{p_2}$.
By assumption, $X_{p_1} \cap N(v) = \emptyset$. Let 
$X_q$ be the first bag in the sequence that contains a
neighbor of $v$. Since $G$ does not have isolated vertices, $N(v)$ is non-empty, 
and hence $p_1 < q < p_2$. 
Let $\widehat v \in X_q$ be a neighbor of $v$.
We have two cases.

\noindent \textbf{Case 1:} $|N(\widehat v) \cap X_q| > 1$. That is, $\widehat v$ 
has other neighbors in $X_q$ apart from $v$.
We consider the following modified
sequence $\mathcal P'$: 
$$X_1, \dots, X_{p_1 - 1}, X_{p_1 + 1} \backslash \{v\}, X_{p_1 + 2} \backslash \{v\}, \dots, X_{q-1} \backslash \{v\}, X_q \backslash \{v\},
X_q, X_{q+1}, \dots, X_s$$
That is, we delay the introduction of $v$ till its first neighbor, $\widehat v$, has 
been introduced. It can be verified that the above sequence $\mathcal P'$ is a 
path decomposition of the same graph $G$ with width no more than the 
width of $\mathcal P$. In the new sequence $\mathcal P'$, $v$ is introduced by $X_q$ which is not a violating bag in $\mathcal P'$.
Below, we explain that the fix-up process has not introduced any new violations. 
Since $v$ sees $\widehat v$ in $X_q$ for the first time, it follows 
that the bag $X_q$ introduces $\widehat v$ in $\mathcal P$.

We first note that $X_q$ was not a special bag in $\mathcal P$. 
To see why, let us assume the contrary. Let $X_q$ be a special 
bag in $\mathcal P$ that introduces the vertices $\widehat v$ and  $w$.
If so, $N(\widehat v) \cap X_q = \{w\}$ as per SN3. Hence $v \notin X_q$.
So we can conclude that $X_q$ is not a special bag, and therefore 
$X_q$ introduces just one vertex $\widehat v$ as per SN1.
This means that $v$ cannot have any other neighbors in $X_q$ apart 
    from $\widehat v$. Hence $\widehat v$ is the only vertex that
    loses a neighbor from its introducing bag due to the fix-up process.
    However, since $\widehat v$ has other neighbors in $X_q$ apart from 
    $v$, it does not result in a violation in $\mathcal P'$.


\noindent \textbf{Case 2:} $|N(\widehat v) \cap X_q| = 1$. That is, $v$ is the
lone neighbor of $\widehat v$ in $X_q$. Since $X_q$ is the first bag in $\mathcal P$ 
that contains a neighbor of $v$, it follows that $X_q$ introduces $\widehat v$.
Since $v$ is the lone neighbor of $\widehat v$ in $X_q$, it follows that $X_q$ is not a special bag in $\mathcal P$.
Hence $\widehat v$ is the lone neighbor of $v$ in $X_q$.
We consider the following sequence $\mathcal P'$:
$$X_1, \dots, X_{p_1 - 1}, X_{p_1 + 1} \backslash \{v\}, X_{p_1 + 2} \backslash \{v\}, \dots, X_{q-1} \backslash \{v\}, 
X_q, X_{q+1}, \dots, X_s$$
We introduce $v$ together with $\widehat v$ in the bag $X_q$. We have already seen that
$N(v) \cap X_q = \{\widehat v\}$ and $N(\widehat v) \cap X_q = \{ v\}$. Thus $X_q$ becomes
a special bag in $\mathcal P'$ that introduces $v$ and $\widehat v$. 
No other violations have been introduced by this because $v$ does not
have any neighbors in $X_{p_1 +1}, X_{p_1 +1}, \ldots, X_{q-1}$.

Thus by repeating this fix-up process for each of the violations, we can convert
the given nice path decomposition into a semi-nice path decomposition.
\qed

\section{Feedback Vertex Set}\label{sec:fvs}


\begin{definition}[Feedback Vertex Set]
Let $G=(V,E)$ be an undirected graph. 
A \emph{feedback vertex set (FVS)} is a set of vertices $S \subseteq V$, removal of which from the graph $G$ makes the
remaining graph ($G[V\setminus S]$) acyclic. 
The size of a smallest such set $S$ is denoted as ${\sf fvs}(G)$.
\end{definition}

\begin{theorem}\label{thm:fvs}
$\chicf \leq {\sf fvs}(G) + 2$. 
\end{theorem}

The following graph (as observed in \cite{Boe2019}), shows that the above theorem
is tight.
Let $K^*_n$ be the graph obtained by starting with 
the clique on $n$ vertices, and subdividing each edge 
with a vertex.
Then $K^*_n$ has an 
FVS of size $n-2$, and it can be seen that  $\chi_{ON}(K^*_n) = n$. 

The proof of this theorem is through a constructive process to \cf{} color the vertices 
of the graph $G$, given a feedback vertex set $F$ of $G$.
By definition, $G[V\setminus F]$ is a collection of trees.

Each tree $T$ in $G[V\setminus F]$ is rooted at 
an arbitrary vertex $r_T$.
If $|V(T)|\geq 2$, we choose a neighbor of $r_T$ and call it the \emph{special vertex} in $T$, denoted by $s_T$. 
Let $v$ be a vertex not in $T$. 
The \emph{deepest neighbor} of $v$ in $T$, denoted by $deep_T(v)$, is a 
vertex $w\in V(T)\cap N(v)$ such that $dist_T(r_T, w)$ is maximized.
If there are multiple such vertices at the same distance, the 
deepest neighbor is chosen to be a vertex which is not the special vertex $s_T$.

\begin{lemma}\label{lem:tree2col}
Let $T$ be a tree with $|V(T)| \geq 2$. 
Then $\chi_{ON}(T) \leq 2$.
\end{lemma}
\begin{proof}
We assign colors $C:V(T) \rightarrow \{1, 2\}$ in the following manner. 
\begin{itemize}
    \item Assign $C(r_T)=1$ and $C(s_T)=2$. 
    \item For each vertex $v\in N_T(r_T)\setminus \{s_T\}$, assign $C(v)=1$. 
    \item For the remaining vertices $v \in V(T)$, assign $C(v)=\{1, 2\}\setminus C(w)$, 
    where $w$ is the grandparent of $v$. 
\end{itemize}
For each vertex $v\in V(T)\setminus \{r_T\}$, the uniquely colored neighbor is its parent. For $r_T$, the uniquely colored neighbor is $s_T$. This is a \cf{} 2-coloring of $T$. 
\qed
\end{proof}

%


We first prove a special case of Theorem \ref{thm:fvs}.
\begin{lemma}\label{lem:whenfeq1}
Let $G=(V,E)$ be a graph and $F\subseteq V$ be a feedback vertex set with $|F|=1$. Then $G$ can be \cf{} colored using 3 colors. 
\end{lemma}

\begin{proof}
Let $F=\{v\}$. First using Lemma \ref{lem:tree2col}, we color all the trees $T \subseteq G[V \setminus F]$ using the colors 2 and 3, whenever $|V(T)| \geq 2$. All the singleton
components of $G[V \setminus F]$ are assigned the color 2. We assign $C(v) = 1$. 
Now all the vertices, except possibly $v$, have a uniquely colored
neighbor.
We explain how to fix this and obtain a \cf{} 
coloring.

\vspace{-0.05in}
\begin{itemize}
    \item \textbf{Case 1: There exists a singleton component $\{w\} \subseteq G[V\setminus F]$.}
    
    Reassign $C(w)=1$. 
    
    \item \textbf{Case 2:  Else, if there exists a component $T\subseteq G[V\setminus F]$, such that 
    either (i) $deep_T(v)\neq s_T$ or (ii) $deep_T(v)=s_T$ and $\{r_T, v\} \notin E(G)$.}
    
    Reassign $C(deep_T(v))=1$. 
    
    \item \textbf{Case 3: Else, for each component $T\subseteq G[V\setminus F]$, $N(v) \cap V(T) = \{r_T, s_T\}$.}
    
    If there exists a component $T\subseteq G[V\setminus F]$, such that $|V(T)|\geq 3$, 
    choose a vertex $w\in V(T)\setminus \{r_T,s_T\}$ and set $w$ as the new root of $T$. Reassign $s_T$ and the colors  of $V(T)$ accordingly.
    Doing so will ensure that $deep_T(v)\neq s_T$.
    We apply Case 2. 
    
    Else, for all the components $T\subseteq G[V\setminus F]$, 
    we have $|V(T)|=2$. 
    Choose a component $T'\subseteq G[V\setminus F]$.
    For all the other vertices $w \in V \setminus (\{v\} \cup V(T'))$, reassign $C(w)= 2$.
\end{itemize}
\vspace{-0.05in}
All the trees in $G[V\setminus F]$ are \cf{} colored as per the earlier 
described procedure. Even after reassigning some colors, they remain
\cf{} colored. 
The vertex $v$ sees another vertex $w$, with $C(w) = 1$ if in Case 1 or 2. In the last case, $v$ sees a unique vertex that 
is colored 3. 
\qed
\end{proof}


\subsection{Proof of Theorem \ref{thm:fvs}}
When $|F|=1$, three colors are sufficient to \cf{} color $G$ by Lemma~\ref{lem:whenfeq1}. 
Now, we consider the case when 
$|F| \geq 2$. 
We assign colors $C:V(G)\rightarrow [|F|+2]$ in such a way that 
$G$ is \cf{} colored.
First by Lemma \ref{lem:tree2col}, we color all the components $T\subseteq G[V\setminus F]$ with $|V(T)| \geq 2$, 
using the colors $|F|+1$ and $|F|+2$. 

During the algorithm, we will keep track of a color which when assigned to the isolated vertices in the feedback vertex set  
does not change the unique color in the neighborhood 
of the already colored vertices. We call this color a \emph{free color} and denote this
by $c'$, initialized to 0.

Let $F = \{v_1, v_2, \dots, v_{|F|}\}$. 
Let $Y= \{v_i \in F : \mbox{deg}_F (v_i)\geq 1\}$. 
For each $v_i\in Y$, assign $C(v_i)=i$. Note that for 
every $v_i\in Y$, there is at least one uniquely colored neighbor in $Y$. 
If $Y\neq \emptyset$, choose an arbitrary vertex $v_i \in Y$ and set $c' = i$.

Now the vertices of $Y$ and components $T\subseteq G[V\setminus F]$ with $|V(T)| \geq 2$ are colored and have a uniquely colored neighbor. What remains are the 
 vertices in $F\setminus Y$ and singleton components of $G[V\setminus F]$. Below 
 we explain how to color them in phases.
 
Recall that $G[F\setminus Y]$ is an independent set. 

\noindent \textbf{Case 1: Singleton component $\{w\} \subseteq G[V\setminus F]$ 
    where $w$ has at least 1 uncolored neighbor.} 
    \begin{itemize}
        \item Let $v_{i_1}, v_{i_2}, \dots , v_{i_m} \in F\setminus Y$ be the uncolored neighbors of $w$, where $m\geq 1$. 
        \item 
        Assign $C(v_{i_1})=C(w)=i_1$ and $C(v_{i_j})=i_2$ for all $2\leq j\leq m$. 
        The uniquely colored neighbor for $w$ is $v_{i_1}$ and for all $v_{i_j}$, it is the vertex $w$. 
        \item The free color is set as $c'=i_1$.
    \end{itemize}
    
\noindent \textbf{Case 2: Singleton component $\{w\}\subseteq G[V\setminus F]$, 
    where all of $N(w)$ is colored and $w$ has no uniquely colored neighbor.} 
    
    \noindent This means that $N(w) \geq 2$, and every color in $N(w)$ appears at least twice.
     Choose two vertices $v_{i_1}, v_{i_2}\in N(w)$ such that $C(v_{i_1}) = C(v_{i_2})$. It must be the case that at least one of the colors $i_1, i_2$ does 
    not appear in $N(w)$. Without loss of generality, let it be $i_1$.
    \begin{itemize}
    \item Reassign $C(v_{i_1}) = i_1$. Assign $C(w) = |F|+ 1$.
        The uniquely colored neighbor for $w$ is $v_{i_1}$. Notice that 
        all the vertices in $N(w)$ would have received their uniquely colored
        neighbors when they were assigned a color.
        \item The free color is set  as $c'=i_1$.
    \end{itemize}

Now, all the singleton components $\{w\} \subseteq V\setminus F$ have uniquely colored neighbors, but not all of them may be colored.
Assign all the uncolored singleton components the color $|F|+1$. 
 What remains to be addressed are the 
remaining uncolored vertices in $F\setminus Y$. These vertices 
do not have any 
singleton components of $G[V\setminus F]$ as neighbors. 
We assign colors if the below Cases 3 or 4 apply.

\noindent \textbf{Case 3: A component $T\subseteq G[V\setminus F]$ such that $s_T$ has at least two uncolored vertices in $F\setminus Y$ as neighbors.}
     \begin{itemize}
       \item Let $v_{i_1},v_{i_2}, \dots , v_{i_m} \in F\setminus Y$ 
       be the uncolored neighbors of $s_T$, with $m\geq 2$. 
        \item Reassign $C(s_T)=i_1$ and assign $C(v_{i_j})=i_2$ for all $1 \leq j\leq m$. The vertex $s_T$ serves as the uniquely
        colored neighbor for the vertices $v_{i_j}$.
        \item The free color is set  as $c'=i_2$.
    \end{itemize}
   

\noindent \textbf{Case 4: There exists an uncolored vertex 
$v_i$ and component $T\subseteq G[V\setminus F]$, such that either (i) $deep_T(v_i)\neq s_T$
    or  (ii)  $deep_T(v_i)=s_T$ and $\{r_T,v_i\} \notin E(G)$.} 
     \begin{itemize}
       \item Reassign $C(deep_T(v_i))=i$ and assign $C(v_i) = i$.
       The vertex $deep_T(v_i)$ serves as the uniquely colored
       neighbor for $v_i$.
       \item The free color is set as $c'=i$.
    \end{itemize}

\noindent \textbf{Case 5 : There exists an uncolored vertex $v_i$ such that for each component $T\subseteq G[V\setminus F]$, either 
$N(v_i) \cap V(T) = \{r_T, s_T\}$ or $N(v_i) \cap V(T) = \emptyset$.}
\noindent We make use of the free color $c'$ obtained from the previous cases. In this case, we reassign $C(s_T)=i$ and assign $C(v_i)=c'$. 
The vertex $s_T$ will serve as the uniquely colored neighbor for $v_i$.

Now we explain why we must have a non-zero free color. If $c'=0$, we have 
that $Y = \emptyset$ and none of the previous cases have been applicable. That is: 
       \begin{enumerate}
            \item $Y = \emptyset$.
           \item There are no singleton components in $G[V\setminus F]$. 
           \item For each vertex $v_i \in F$ and 
           for each component $T\subseteq G[V\setminus F]$, either 
            $N(v_i) \cap V(T) = \{r_T, s_T\}$ or $N(v_i) \cap V(T) = \emptyset$.
            \item  For each component $T\subseteq G[V\setminus F]$, 
            $|N(s_T) \cap F| \leq 1$.
           \end{enumerate}
     Since $|F|\geq 2$, let us consider $v_1, v_2 \in F$. Notice 
     that due to the above, it is not possible to have a path from
     $v_1$ to $v_2$ in $G$. This means that $G$ is not connected.
     This is a contradiction. Thus we must have $c' \neq 0$.
     
     We have described a procedure to obtain a \cf{} coloring 
     that uses $|F| +2$ colors. By setting $F$ to be a minimum
     sized FVS, we get $\chicf \leq {\sf fvs}(G) + 2$.
     \qed

\section{Neighborhood Diversity \& Distance to Cluster}\label{sec:op}

In this section, we give
improved bounds for $\chion$ and $\chicn$ with respect to the parameters
neighborhood diversity and distance to cluster.

\subsection{Neighborhood Diversity}\label{sec:nd}

\begin{definition}[Neighborhood Diversity \cite{gargano2015}]
Give a graph $G = (V, E)$, two vertices $v, w \in V$ have the same \emph{type}
if $N(v) \setminus \{w\} = N(w) \setminus \{v\}$. A graph $G$ has \emph{neighborhood
diversity} at most $t$ if $V(G)$ can be partitioned into $t$ sets $V_1, V_2, \dots, 
V_t$, such that all the vertices in each $V_i$, $1 \leq i \leq t$ have the 
same type.
The partition $\{V_1, V_2, \dots, V_t\}$ is called the \emph{type partition} of $G$.
\end{definition}
It can be inferred from the above definition that all vertices in a $V_i$ 
either form a clique or an independent set, $1\leq i\leq t$. For two types
$V_i, V_j$, either each vertex in $V_i$ is neighbor to each vertex in $V_j$,
or no vertex in $V_i$ is neighbor to any vertex in $V_j$. This leads to the 
definition of the \emph{type graph} $H= (\{1, 2, \dots, t\}, E_H)$, where
$E_H = \{\{i, j\} : 1\leq i < j \leq t, \mbox{ each vertex in } V_i \mbox{ is a neighbor of each vertex in }V_j \}$.

In the above, $cl(G)$ and $ind(G)$ respectively denote the number of $V_i$'s that form a clique and independent set in the type partition $\{V_1, V_2, \dots, V_t\}$. 
Gargano and Rescigno~\cite{gargano2015} 
showed that the CFON and CFCN variants are fixed parameter tractable with respect to 
neighborhood diversity.
They also obtained the bounds
$\chi_{ON}(G)\leq \chi_{ON}(H) + cl(G) +1$ and 
$\chi_{CN}(G)\leq \chi_{CN}(H) + ind(G) +1$. 
We improve both these bounds.
\begin{theorem}\label{thm:ndopen}
$\chicf\leq \chi_{ON}(H) + \frac{cl(G)}{2} +2$. 
\end{theorem}

\begin{theorem}[$\star$]\label{thm:ndclosed}
$\chi_{CN} (G)\leq \chi_{CN}(H) + \frac{ind(G)}{3} +3$. 
\end{theorem}

We prove Theorem \ref{thm:ndopen} below. The proof of Theorem 
\ref{thm:ndclosed} uses similar ideas and is presented in Appendix \ref{sec:ndproofs}.

\begin{proof}[Proof of Theorem \ref{thm:ndopen}]
To begin with, we {\cf} color the type graph $H$ using 
$\chi_{ON}(H)$ colors. Let $C_H: V_H \rightarrow[\chi_{ON}(H)]$ be that coloring and  $U_H: V_H \rightarrow[\chi_{ON}(H)]$ be the corresponding
assignment of unique colors.
Now, we derive a coloring $C: V(G) \rightarrow \{0, 1, 2, \dots, s, s+1\}$ from $C_H$, where 
$s=\chi_{ON} (H) + \frac{cli(G)}{2}$. 
Also we identify a unique color in the neighborhood of each vertex, denoted
$U:V(G)\rightarrow \{0, 1, 2, \dots, s, s+1\}$.
Let $V_1,V_2, \dots , V_t$ be the type partition of $V$. We assign colors to the vertices as follows: For each $V_i$, choose a representative vertex $r_i \in V_i$ and assign $C(r_i)=C_H(i)$. For each $V_i$, for all vertices $x \in V_i\setminus \{r_i\}$, 
we assign $C(x) = 0$. We make the below observations.
\begin{itemize}
    \item 
    Each of the vertices $r_i$ has a uniquely colored neighbor, as $C_H$ is a {\cf} coloring of $H$. 
    \item Let $V_i$ be an independent set, 
    and let $r_j$ be the uniquely colored neighbor of $r_i$. For each $x\in V_i$, $r_j$ serves as the uniquely colored neighbor. 
    
    \item If $V_i$ is a clique and $U_H(i)\neq C_H(i)$, 
    the uniquely colored neighbor of $r_i$ serves as the uniquely colored neighbor for all 
    vertices in $V_i$.
\end{itemize}

What remains to be handled are the type sets $V_i$ which are cliques and $U_H(i)=C_H(i)$. 
We call these type sets as \emph{bad sets}. 
We do not consider the singleton $V_i$'s as bad sets.
All the representative vertices $r_i$ see a uniquely colored neighbor, regardless of whether $V_i$ is bad or not.  
Note that, once a bad set $V$ is fixed, we no longer call it a bad set. 

Let $A$ refer to the following set of colors: $A = \{\chi_{ON}(H) + 1,\chi_{ON}(H) + 2, \dots, \chi_{ON}(H) + cli(G)/2\}$. None of the colors from $A$ have been used till now.

\noindent \textbf{Reduction of bad sets:} If there exists a $V_i$ (not necessarily a bad set) 
that has at least 2 bad sets as neighbors, we do the following. 
Let $V_{i_1},V_{i_2}, \dots, V_{i_m}$ be the bad sets adjacent to $V_i$. Then we reassign
$C(r_i) = c$, where $c \in A$ is a color that has not been used till now. 
The vertex $r_i$ will serve as the uniquely colored neighbor for all vertices in 
$V_{i_1}, V_{i_2}, \dots, V_{i_m}$ as well as the vertices in $V_i \setminus \{r_i\}$.
Thus after this operation, none of $V_{i_1},V_{i_2}, \dots, V_{i_m}$ and $V_i$ are bad sets.

We apply the above reduction operation as much as possible, choosing a new color
from $A$ each time. After that, each set $V_i$ is adjacent to at most one bad set. 
This leaves us with the following two cases. 
\begin{itemize}
    \item \textbf{Case 1: Bad sets $V_i$ and $V_j$ which are neighbors, each of which is not 
    neighbors to any other bad sets.}
    
    Reassign $C(r_i)=s+1$.
    The uniquely colored neighbor of $r_i$ remains the same. And $r_i$ becomes the uniquely colored neighbor for all vertices $x\in V_i\cup V_j \setminus \{r_i\}$. 
    
    Note that any set $V_k$ that relied on $V_i$ for its unique color,
    can continue to do so. This is because $V_k$ sees at most one bad set 
    after the repeated application of the reduction operation.
    
    \item \textbf{Case 2: Bad set $V_i$, which has no neighboring bad set.} 
    
    Let $V_j$ be the neighboring set of $V_i$  such that $C(r_i)=C(r_j)$. 
    We reassign     $C(r_i)=s+1$. 

Every vertex in $V_i$ has $r_j$ as its uniquely colored neighbor. As in the previous case, any set that
relied on $V_i$ for its unique color can continue to do so. 
\end{itemize}
The above is a \cf{} coloring. We use $\chi_{ON}(H)$ colors to color the representative vertices of each $V_i$. Each application of the reduction operation needs
one new color  from $A$ to handle at least two bad sets. Since each bad set is a 
clique, the number of extra colors needed is at most $cl(G)/2$. Taking the colors
$\{0, s+1\}$ into account, the total number of colors used is $\chi_{ON}(H) + 
cl(G)/2 + 2$. 
\qed
\end{proof}

\subsection{Distance to Cluster}\label{sec:dc}

\begin{definition}[Distance to Cluster]
Let $G= (V,E)$ be a graph. The distance to cluster of $G$, 
denoted ${\sf dc}(G)$, is the size of the smallest set $X\subseteq V$  
such that $G[V\setminus X]$ is a disjoint union of cliques. 
\end{definition}

Reddy \cite{vinod2017}, studied the CFCN and the CFON varaints 
with respect to 
the distance to cluster parameter, 
${\sf dc}(G)$.
They showed that 
$\chi_{ON}(G)\leq 2{\sf dc(G)} +3$ and 
$\chi_{CN}(G)\leq {\sf dc}(G)+2$. 
We give the 
following improved bounds. 

\begin{theorem}\label{thm:dist2clust}
$\chicf \leq {\sf dc}(G) + 3$. 
\end{theorem}

\begin{theorem}[$\star$]\label{thm:cfcnd+1}
$\chicn{}\leq \max \{3,{\sf dc}(G)+1\}$.
\end{theorem}
For the subdivided clique
$K^*_n$, we have $\chi_{ON}(K^*_n) = {\sf dc}(K^*_n) = n$.
Hence Theorem \ref{thm:dist2clust} is nearly tight.
We prove Theorem \ref{thm:dist2clust} below. The proof of Theorem 
\ref{thm:cfcnd+1} uses similar ideas and is presented in Appendix 
\ref{sec:dcproofs}.

\begin{proof}[Proof of Theorem \ref{thm:dist2clust}]
Let ${\sf dc}(G) = d$. That is, there is a set $X \subseteq V(G)$, with 
$|X| = d$ such that
$G[V \backslash X]$ is a disjoint union of cliques.

If $X=\emptyset$, the graph $G$ is a clique because we only 
consider connected graphs. 
A clique can be \cf{} colored using 3 colors.
Else, we have $|X| \geq 1$. Let $X = \{v_1, v_2, \ldots, v_d\}$.
Then $G[V\backslash X] = K_1 \cup K_2 \dots \cup K_t$ is a disjoint union of cliques.

Below, we explain how to assign colors, $C: V(G) \rightarrow [d+3]$ such that
every vertex has a uniquely colored neighbor.
We apply the following rules:
    \begin{enumerate}
        \item Let $Y = \{v_i \in X : \mbox{deg}_X 
        (v_i)\geq 1\}$. For all $v_i \in Y$, assign $C(v_i)=i$. 
        
        Now every vertex in $Y$ is colored and has a uniquely colored neighbor. 
        
        \item For each of the singleton cliques $K_j=\{w\}$, we do the following. 
        \begin{itemize}

        \item \textbf{Case 2(a): The vertex $w$ has at least 1 uncolored neighbor}. 
 
        Let $v_{i_1}, v_{i_2}, \dots , v_{i_m} \in X$ be the uncolored neighbors of $w$, with $m\geq 1$. 
        Assign $C(v_{i_1})=C(w)=i_1$ and $C(v_{i_\ell})=d+1$, for all $2\leq \ell \leq m$. 
         All the vertices in $N(w) \cup \{w\}$ see the color $i_1$ 
         exactly once in their neighborhood. We will not be assigning
         the color $i_1$ for any other vertices henceforth.
         
        \item \textbf{Case 2(b): All vertices in $N(w)$ are colored}.
        
        The assignment of colors in the previous case may lead us to this 
        case. If $w$ already sees a uniquely colored neighbor, then we 
        set $C(w) = d+1$.
        
        If $w$ has no uniquely colored neighbor, we choose 
        two vertices $v_{i_1}, v_{i_2} \in N(w)$ such that $C(v_{i_1})=C(v_{i_2})$. 
        Since the only color that is being reused in $X$ is $d+1$, 
        we have $C(v_{i_1})=C(v_{i_2}) = d+1$. Reassign $C(v_{i_1})=i_1$.
        Assign $C(w)=d+1$. Here the color $i_1$ will be the unique color in the neighborhood of $w$ and this color will not be used in further coloring. 
        \end{itemize}
        
        After this step, all the singleton cliques $K_j$ and their neighboring vertices are colored and also have a uniquely colored neighbor.
        \item For each uncolored $v_i \in X\setminus Y$ that does not have a uniquely colored neighbor, 
        we choose a vertex $w\in N(v_i)$ and assign $C(w)=i$. The color $i$ is the unique color in $v_i$'s neighborhood. And the color $i$ is not used in further coloring. 
        \item For all the remaining uncolored $v_i \in X\setminus Y$, 
        assign  $C(v_i) = d+1$. 
        Recall that 
        $v_i$ is not colored in Step 3 because it has a uniquely colored neighbor. 

        Now, all the vertices in $X$ are colored and have uniquely colored neighbors. 
        What remains to be colored are the cliques of size at least 2. 
         
     \item For each clique $K_j$ with $|K_j| \geq 2$, we note that there may already be some colored vertices in $K_j$ 
    as a result of Step 3. These colors appear exactly once in the graph.
    We do the following:
    \begin{itemize}
         \item If $K_j$ has at least 2 colored vertices, color the remaining vertices (if any) with $d+1$. 
          \item Else, if $K_j$ has exactly 1 colored vertex, choose an uncolored vertex and assign $d+2$. 
          Color the remaining vertices (if any) with $d+1$. 
          \item Else, choose 2 vertices from $K_j$ and assign the colors $d+2$ and $d+3$. 
          Color the remaining vertices (if any) with $d+1$. 
     \end{itemize}
     \end{enumerate}
\qed
\end{proof}

\section{\partialcf{} Coloring of Planar Graphs}\label{sec:planar}
\begin{definition}[Planar and Outerplanar graphs]
A \emph{planar graph} is a graph that can be drawn in $\mathbb R^2$ (a plane) such 
that the edges do not cross each other in the drawing.
An \emph{outerplanar graph} is a planar graph that has a drawing in a plane such 
that all the vertices of the graph belong to the outer face. 
\end{definition}

Abel et. al. showed \cite{planar} that eight colors are sufficient for 
\partialcf{} coloring of a planar graph. 
In this section, we improve the bound to five colors. 


We need the following definition:
\begin{definition}[Maximal Distance-3 Set]
For a graph $G= (V,E)$, a \emph{maximal distance-3 set} is a set $S \subseteq V(G)$
that satisfies the following:
\begin{enumerate}
    \item For every pair of vertices $w, w' \in S$, we have $dist(w,w') \geq 3$.
    \item For every vertex $w \in S$,  $\exists w' \in S$ such that
            $dist(w, w') = 3$.
    \item For every vertex $x \notin S$, $\exists x' \in S$ such 
            that $dist(x, x') < 3$.
\end{enumerate}
\end{definition}
The set $S$ is constructed by initializing $S=\{v\}$ where $v$ is 
an arbitrary vertex. We proceed in iterations. In each iteration, we add 
a vertex $w$ to $S$ if (1) for every $v$ already in $S$, $dist(v,w) \geq 3$,
and (2) there exists a vertex $w'\in S$ such that $dist(w,w') = 3$.
We repeat this 
until no more vertices can be added. 


The main component of the proof is the construction of an auxiliary graph 
$G'$ from the given graph $G$. 

\vspace{1mm}
\noindent\textbf{Construction of $G'$:} 
The first step is to pick a maximal distance-3 set $V_0$.
Notice that any distance-3 set is an independent set by definition.
We let $V_1$ denote the neighborhood of $V_0$. More formally,
$V_1 = \{ w :  \{w, w'\} \in E(G), w' \in V_0\}$.
Let $V_2$ denote the remaining vertices i.e., 
$V_2=V\setminus (V_0 \cup V_1)$. 

We note the following properties satisfied by the above partitioning of $V(G)$.
\begin{enumerate}
    \item The set $V_0$ is an independent set.
    \item For every vertex $w \in V_1$, there exists a unique vertex $w' \in V_0$ such that $\{w, w'\} \in E(G)$. This is because if there are two such vertices, this will violate
    the distance-3 property of $V_0$.
    \item Every vertex in $V_0$ has a neighbor in $V_1$. If there exists $v \in V_0$
            without a neighbor in $V_1$, then $v$ is an isolated vertex. By assumption,
            $G$ does not have isolated vertices.
    \item There are no edges from $V_0$ to $V_2$.
    \item Every vertex in $V_2$ has a neighbor in $V_1$, and is hence at distance 2 from some vertex in $V_0$. This is due to the maximality of the distance-3 set $V_0$. 
\end{enumerate}

Now we define $A=V_0\cup V_2$. 
We first remove all the edges of $G[V_2]$ making $A$ an independent set.
For every vertex $v\in A$ we do the following: we identify an arbitrary neighbor
$f(v) \in N(v) \subseteq V_1$. Then we contract the edge $\{v, f(v)\}$.
That is, we first identify vertex $v$ with $f(v)$. 
Then for every edge $\{v, v'\}$, we add an edge $\{f(v),v'\}$.
The resulting graph is $G'$. 


\begin{theorem}\label{graphon}
If $G$ is a planar graph, $\chi^*_{ON}(G) \leq 5$.
\end{theorem}
\begin{proof}
Let $G$ be a planar graph. 
We first construct the graph $G'$ as above.
Since the steps for constructing $G'$ involve only edge deletion 
and edge contraction, $G'$ is also a planar graph.
By the planar four-color theorem \cite{4ct}, 
there is an assignment $C: V(G') \rightarrow \{2,3,4,5\}$
such that no two adjacent vertices of $G'$ are assigned the same color.
Now we have colored all the vertices in $V(G') = V_1$

Now, we extend $C$ to get a \partialcf{} 
coloring for $G$. 
For all vertices $v \in V_0$, we assign $C(v) = 1$.
The vertices in $V_2$ are not assigned a color.

We will show that $C$ is indeed a \partialcf{} 
coloring of $G$. 
Consider a vertex $v\in A$ which is contracted to a neighbor $f(v)=w\in V_1$. 
The color assigned to $w$ is distinct from all $w$'s neighbors
in $G'$. 
Hence the color assigned to $w$ is the unique color 
among the neighbors of $v$ in $G$.

For each vertex $w \in V_1$, $w$ is a neighbor of exactly one vertex $v\in V_0$.
Every vertex $v \in V_0$ is colored 1, which is different 
from all the colors assigned to the neighbors of $w$ in $G'$.
\qed
\end{proof}


Outerplanar graphs have a proper coloring using three colors. By 
argument analogous to Theorem \ref{graphon}, we infer the following.

\begin{corollary}\label{cor:outer}
If $G$ is an outerplanar graph, $\chi^*_{ON}(G) \leq 4$.
\end{corollary}

For outerplanar graphs, a \partialcf{} coloring using 4 colors implies
a \cf{} coloring using 5 colors. However, we can show the 
following improved bound. For the sake of clarity, we provide the full proof of
the below theorem in Appendix \ref{sec:outer} and give a sketch below.

\begin{theorem}[$\star$]\label{thm:4col}
If $G$ is an outerplanar graph, $\chi_{ON}(G) \leq 4$.
\end{theorem}

\begin{proof}[Proof Sketch of Theorem \ref{thm:4col}]
Theorem \ref{thm:4col} is proved using a two-level induction process. The first level
is using a \emph{block decomposition} of the graph. Any connected graph can be viewed
as a tree of its constituent blocks. We color the blocks in order so that when we
color a block, at most one of its vertices is previously colored. Each block is colored 
without affecting the color of the already colored vertex. The second level 
of the induction is required for coloring each of the blocks. We use
ear decomposition on each block and color the faces of the block in sequence.
However, the proof is quite technical and involves several cases of analysis at each 
step.
\qed
\end{proof}



\noindent \textbf{Acknowledgments:} We would like to thank
I. Vinod Reddy for suggesting the problem, 
Rogers Mathew and N. R. Aravind for helpful discussions and  
the anonymous reviewer who pointed
out an issue with the proof of Theorem \ref{thm:pwneighbor}.

\bibliographystyle{elsarticle-num}

\bibliography{BibFile}
\appendix

\section{Proofs of Cases 2 and 3 of Theorem \ref{thm:technical}}
\subsection{Proof of Case 2}\label{sec:case2pw}
$X$ is a bag that introduces one vertex $v_k$ and  $2k - 1 \in  \{U(x): x \in X'\}$. 
This means that 
    $F(X') \setminus 
     \{U(x): x \in N(v_k) \cap X'\} \neq \emptyset$.
    In this case, $C(v_k) = U(w) = 2k-1$, 
    for a vertex $w$ that is needy in $X'$, chosen according to Rule 1. 
    
    If $C(w) \notin \{1, 2, 3, \ldots, 2k - 2\}$,
    then $(S \cup \{w\}) \setminus \{v_k\}$ is an expensive subset of size $k$ in $X'$, the predecessor of $X$. This contradicts the
    choice of $X$. Hence $C(w) \in \{1, 2, 3, \ldots, 2k - 2\}$. 
    By invariant 1, for any $v, v' \in X'$, we have $C(v) \neq C(v')$.
    We can rule out the colors $\{1, 3, \dots, 2k-3\}$ since they appear
    as $C(v_i)$, for $1\leq i \leq k-1$. We can also rule out the colors
    $\{2, 4, \dots, 2r\}$ since the vertices $v_i$ are needy for $1 \leq i \leq r$. Hence $C(w) \in \{2(r+1), 2(r+2), \ldots, 2k - 2\}$. 
    Let $C(w)  = 2j$ for some $r+1 \leq j \leq k-1$. Since  
    $U(v_j) = 2j$, we have  $|\{x : x\in X', C(w) = U(x)\}| \geq 1$. 
    
     Without loss of generality,
    let $\{2, 4, \ldots, 2\ell\} \subseteq 
    \{U(x): x\in N(v_k) \cap X'\}$
    and $\{2(\ell+1), \ldots, 2r\} \cap 
    \{U(x): x\in N(v_k) \cap X'\}  = \emptyset$
    for some $0 \leq \ell \leq r$. Clearly, we cannot choose $C(v_k)$ from 
    $\{2, 4, \ldots, 2\ell\}$.  Let us try to understand why $2k-1$ was chosen as 
    $C(v_k)$ over elements of $\{2(\ell+1), \ldots, 2r\}$. We have two subcases.
    \begin{itemize}
        \item $2k-1 \in F_1(X')$.
    Then $w$ is the lone vertex  in $X'$ such that $U(w) = 2k - 1$.  
    Without loss of generality, let colors $2(\ell +1), \ldots, 2\ell' \in F_1(X')$ and 
    let $2(\ell' + 1), \ldots, 2r \in F_{>1}(X')$ for some $\ell < \ell' \leq r$.
    
    As per Rule 1, $2k-1$ was chosen as a color $c \in F_1(X') \setminus
    \{U(x): x\in N(v_k) \cap X'\}$ 
    that minimizes  $|\{x : x\in X', 
    C(U^{-1}(c)) = U(x)\}|$. Since $|\{x : x\in X', C(w) = U(x)\}| \geq 1$,
    for each $\ell + 1 \leq i \leq \ell'$, we have 
    $|\{x : x\in X', C(U^{-1}(2i)) = U(x)\}| = 
    |\{x : x\in X', C(v_i) = U(x)\}| = |\{x : x\in X',2i - 1 = U(x)\}|  \geq 1$. 
    So there exists a set $W' = \{w_{\ell +1}, \ldots, w_{\ell'}\}$ (disjoint from
    $S$) such that $U(w_i) = 2i - 1$, for each $\ell + 1 \leq i \leq \ell'$.
    
    Since $2(\ell' + 1), \ldots, 2r \in F_{>1}(X')$, we have a set 
    $W'' = \{w_{\ell' + 1}, \ldots, w_r\}$ (disjoint from $S$) such that $U(w_i) = U(v_i) = 2i$,
    for each $\ell' + 1 \leq i \leq r$.  Thus we have $W = W' \cup W''$ such that 
    $|W| = r - \ell$ that is disjoint from $S$.
    
        \item $2k-1 \in F_{>1}(X')$. Since a member of $F_{>1}(X')$ was chosen,
        it follows that $F_1(X') \setminus 
        \{U(x): x\in N(v_k) \cap X'\} = \emptyset$. 
        Hence $\{2(\ell+1), \ldots, 2r\} \subseteq 
        F_{>1}(X')$. So we have a set 
        $W = \{w_{\ell + 1}, \ldots, w_r\}$ (disjoint from $S$) 
        such that $U(w_i) = U(v_i) = 2i$,
        for each $\ell + 1 \leq i \leq r$.  Thus we have $W$ with 
        $|W| = r - \ell$ that is disjoint from $S$.
    \end{itemize}
    
    
    If $\ell = 0$, then $|W| = r$, giving us $|X| \geq |S| + |W \cup Z| \geq 3k/2$. In what follows, we will assume $\ell \geq 1$. That is, 
    $v_k$ has at least one needy neighbor.
    Recall that $\{2, 4, \ldots, 2\ell\} \subseteq F(X') \cap 
     \{U(x): x\in N(v_k) \cap X'\} $.
    For $1 \leq i \leq \ell$, let $v'_i \in N(v_k) \cap X'$ such
    that\footnote{The vertices $v'_i$ may or may not be the same as $v_i$.} 
    $U(v'_i) = 2i$.     
    
    Now let us see how $U(v_k)$ was assigned as $2k$. By Rule 2, 
    $U(v_k)$ is set to $C(y)$  such that $y$ is a needy neighbor that
    minimizes $|\{x : x \in X', U(y) = U(x) \}|$. So $C(y) = 2k$. 
    If $U(y) \notin \{1, 2, \dots, 2k-2\}$, then $(S \cup \{y\}) \setminus 
    \{v_k\}$ is a $k$-expensive subset in $X'$, contradicting the choice of $X$. So $U(y) \in \{1, 2, 3,\ldots, 2k-2\}$. Since $y$ is needy, 
    as in Case 1, we can rule out $\{1, 3, \ldots, 2k - 3\} \cup \{2(r+1), 2(r+2), \ldots, 2k-2\}$. So $U(y) \in \{2, 4, 6,\ldots, 2r\}$. 
    
    Let $U(y) = 2j'$, where $1 \leq j' \leq r$. Since $U(v_{j'}) = 2j'$ as well, we have
    $|\{x : x \in X', U(y) = U(x) \}| \geq 2$.
    All of $v'_1, \ldots, v'_{\ell}$ are needy and neighbors to $v_k$. Since
    $y$ was chosen over these vertices, it follows that there exists a 
    set of vertices $Y = \{y_1, \dots, y_{\ell}\}$, disjoint from $S$ such that $U(y_i) = U(v'_i) = 2i$, for $1\leq i \leq \ell$. 
    
    The sets $W$ and $Y$ are disjoint, but need not be disjoint from $Z$.
    Since $|W \cup Y| = r$ and $|Z| = k - r$, we have $|W \cup Y \cup Z| \geq k/2$.
    Since $W, Y, Z$ are all disjoint from $S$, we have that $|X| \geq |S| + 
    |W \cup Y \cup Z| \geq 3k/2$.
    \qed
\subsection{Proof of Case 3}\label{sec:case3pw}

$X$ is a special bag that introduces $v_k$ and $\widehat v_k$. If $F(X')=\emptyset$, then none of the $k-1$ vertices in $S \cap X'$ are needy in $X'$.
Hence $|X'|\geq 2(k-1)$. This implies that 
$|X|\geq 2(k-1)+2 =2k$ and we are done.  

Else, $|F(X')|\geq 1$.
Let us first note that since $S$ is an expensive subset, so is $S\cup \{\widehat v_k\} \setminus \{v_k\}$.
Since $|F(X')|\geq 1$, at least one of  $C(v_k)$ or $C(\widehat v_k)$ will be chosen from $F(X')$.
Without loss of generality, let 
$v_k$ be a vertex such that $C(v_k) \in F(X')$.
Let $C(v_k)=U(w)=2k-1$, where $w$  is a needy vertex in $X'$, chosen according to Rule 1. 

    If $C(w) \notin \{1, 2, 3, \ldots, 2k - 2\}$,
    then $(S \cup \{w\}) \setminus \{v_k\}$ is an expensive subset of size $k$ in $X'$, the predecessor of $X$. This contradicts the
    choice of $X$. Hence $C(w) \in \{1, 2, 3, \ldots, 2k - 2\}$. 
    By invariant 1, for any $v, v' \in X'$, we have $C(v) \neq C(v')$.
    We can rule out the colors $\{1, 3, \dots, 2k-3\}$ since they appear
    as $C(v_i)$, for $1\leq i \leq k-1$. We can also rule out the colors
    $\{2, 4, \dots, 2r\}$ since the vertices $v_i$ are needy for $1 \leq i \leq r$. Hence $C(w) \in \{2(r+1), 2(r+2), \ldots, 2k - 2\}$. 
    Let $C(w)  = 2j$ for some $r+1 \leq j \leq k-1$. Since  
    $U(v_j) = 2j$, we have  $|\{x : x\in X', C(w) = U(x)\}| \geq 1$. 
    
    Since $\{2, 4, \cdots, 2r\} \subseteq F(X')$, 
    let us see why $2k-1$ was chosen as $C(v_k)$ over these colors. 
    We have two subcases. 
    \begin{itemize}
        \item $2k-1 \in F_1(X')$.
    Then $w$ is the lone vertex  in $X'$ such that $U(w) = 2k - 1$.  
    Without loss of generality, let colors $2, 4, \ldots, 2\ell' \in F_1(X')$ and 
    let $2(\ell' + 1), \ldots, 2r \in F_{>1}(X')$ for some $0 < \ell' \leq r$.
    
    As per Rule 1, $2k-1$ was chosen as a color $c \in F_1(X')$ 
    that minimizes  $|\{x : x\in X', 
    C(U^{-1}(c)) = U(x)\}|$. 
    Since $|\{x : x\in X', C(w) = U(x)\}| \geq 1$,
    for each $1 \leq i \leq \ell'$, we have 
    $|\{x : x\in X', C(U^{-1}(2i)) = U(x)\}| = 
    |\{x : x\in X', C(v_i) = U(x)\}| = |\{x : x\in X',2i - 1 = U(x)\}|  \geq 1$. 
    So there exists a set $W' = \{w_{1}, \ldots, w_{\ell'}\}$ (disjoint from
    $S$) such that $U(w_i) = 2i - 1$, for each $1 \leq i \leq \ell'$.
    
    Since $2(\ell' + 1), \ldots, 2r \in F_{>1}(X')$, we have a set 
    $W'' = \{w_{\ell' + 1}, \ldots, w_r\}$ (disjoint from $S$) such that $U(w_i) = U(v_i) = 2i$,
    for each $\ell' + 1 \leq i \leq r$.  Thus we have $W = W' \cup W''$ such that 
    $|W| = r$ that is disjoint from $S$.
    
        \item $2k-1 \in F_{>1}(X')$. Since a member of $F_{>1}(X')$ was chosen,
        it follows that $F_1(X') = \emptyset$.  
                Hence $\{2, 4,  \ldots, 2r\} \subseteq 
        F_{>1}(X')$. So we have a set 
        $W = \{w_{1}, \ldots, w_r\}$ (disjoint from $S$) 
        such that $U(w_i) = U(v_i) = 2i$,
        for each $1 \leq i \leq r$.  Thus we have $W$ with 
        $|W| = r$ that is disjoint from $S$.
    \end{itemize}
    
    In either case, we have $W$ that is disjoint from $S$ and $|W| = r$. Recall that we also 
    have $Z$ disjoint from $S$, such that $|Z| = k - r$. Thus we get that $|X| \geq |S| + |W \cup Z| \geq 3k/2$.   \qed
\section{Proof of Theorem \ref{thm:ndclosed}}\label{sec:ndproofs}

We \cfcn{} color the type graph $H$ using 
$\chi_{CN}(H)$ colors. Let $C_H : V_H \rightarrow[\chi_{CN}(H)]$ be that coloring and  $U_H: V_H \rightarrow[\chi_{CN}(H)]$ be the corresponding
assignment of unique colors.
Now, we derive a coloring $C: V(G) \rightarrow \{0, 1, 2, \dots, s, s+1, s+2\}$ from $C_H$, 
with $s=\chi_{CN} (H) + \frac{ind(G)}{3}$. 
Also we identify a unique color in the neighborhood of each vertex, denoted
$U:V(G)\rightarrow  \{0, 1, 2, \dots, s, s+1, s+2\}$. 
Let $V_1,V_2, \dots , V_t$ be the type partition of $V$. 
We assign colors to the vertices as follows: 
For each $V_i$, choose a representative vertex $r_i \in V_i$ and assign $C(r_i)=C_H(i)$. For each $V_i$, for all vertices $x \in V_i\setminus \{r_i\}$, 
we assign $C(x) = 0$. We make the below observations.
\begin{itemize}
    \item Each of the representative vertices $r_i$ has a uniquely colored neighbor, as $C_H$ is a \cfcn{} coloring of $H$.
    
      \item If $V_i$ is a clique, let $r_j$ be the uniquely colored neighbor of $r_i$ (note that $r_j$ can be $r_i$ itself). 
    For each $x\in V_i$, $r_j$ serves as the uniquely colored neighbor. 
    
    \item If $V_i$ is an independent set such that $C_H(i)\neq U_H(i)$,  
    the uniquely colored neighbor of $r_i$ is the uniquely colored neighbor for all vertices in $V_i$. 
\end{itemize}

What remains to be handled are the independent sets $V_i$, such that $C_H(i)=U_H(i)$. 
We call these type sets $V_i$ (independent sets) as the \emph{bad sets}. 
We do not consider singleton $V_i$'s as  bad sets.
Also, all the representative vertices $r_i$ see a uniquely colored neighbor, regardless of whether $V_i$ is bad or not.  
Once a bad set $V$ is fixed, we no longer call it a bad set. 



\noindent\textbf{Reduction of bad sets:} We process the bad sets in iterations.  We require at most $ind(G)/3$ iterations. In iteration $\ell$, where $1 \leq \ell \leq ind(G)/3$, if there exists a bad set $V_i$, which has at least two neighboring bad sets, we do the following. We call $V_i$ 
    the \emph{lead set} in this iteration.
    Let the neighboring bad sets of $V_i$ be $V_{i_1}, V_{i_2}, \dots, V_{i_m}$, where $m\geq 2$. 
    Choose a vertex $v_i\in V_i\setminus \{r_i\}$ and reassign $C(v_i)= \chi_{CN}(H) + \ell$. We call $v_i$ as the \emph{lead representative} of this iteration.
    For each vertex $x\in V_i\setminus \{r_i,v_i\}$, reassign $C(x) = \chi_{CN}(H) + \ell+1$. 
    Each vertex  in $V_i$ serves as its own the uniquely colored neighbor.
    For each vertex in $V_{i_p}$, the vertex $v_i\in V_i$ serves as the uniquely colored neighbor. 
    
    Let us see why the reduction operation fixes all the lead sets and all their 
    neighboring 
     bad sets. None of the sets chosen as the lead sets in two different
    iterations    are adjacent, else they could have been considered in the same
    iteration. 
    For bad sets that neighbors multiple lead sets, the uniquely colored neighbor
    is provided by the lead representative that was considered earliest
    in the reduction operation.

    
    

In each of the above iterations, at least 3 bad sets are colored. Hence it suffices
to have $ind(G)/3$ iterations. The number of colors required are $ind(G)/3 + 1$.
After the reduction operations, we are left with the bad sets which have at most one bad set as neighbor. We handle them as follows.

\begin{itemize}
    
    \item \textbf{Case 1: Bad sets $V_i$ and $V_j$ which are neighbor, 
    each of which is not neighbors to any other bad sets.}
    
    We note that the color $\chi_{CN} (H) + \frac{ind(G)}{3} + 1 = s+1$ is 
    possibly used only
    in the last iteration of the reduction operation, but does not serve as a 
    unique color for any of the vertices in the bad sets of that iteration. 
    
    We use the colors $s+1$ and $s+2$ for coloring the bad sets $V_i$ and $V_j$. 
    Choose two vertices $x_i\in V_i\setminus \{r_i\}$ and $x_j\in V_j\setminus \{r_j\}$. 
    Reassign $C(x_i)=s+1$ and $C(x_j)=s+2$. 
    These vertices $x_i$ and $x_j$ serve as uniquely colored neighbors for 
    the vertices in $V_j$ and $V_i$ respectively.
    
    \item \textbf{Case 2: A bad set $V_i$ that has no bad set as neighbor.}
    
    Reassign $C(v)=s+2$, for all vertices $v\in V_i\setminus \{r_i\}$. All the vertices
    in $V_i$ serve as their own uniquely colored neighbors. 
\end{itemize}

The above coloring is a \cfcn{} coloring. We use $\chi_{CN}(H)$ colors to color the representative vertices of each $V_i$. 
and $ind(G)/3 +1 $ colors in the reduction operation.
Taking the colors
$\{0, s+2\}$ into account, the total number of colors used is $\chi_{CN}(H) + 
ind(G)/3 + 3$. 
\qed
\section{Proof of Theorem \ref{thm:cfcnd+1}}\label{sec:dcproofs}
Let ${\sf dc}(G)=d$. 
That is, there is a set $X\subseteq V$, with $|X|=d$ 
such that $G[V\setminus X]$ is a disjoint union of cliques. 
Let $X = \{v_1, v_2, \dots, v_d\}$ and 
$Y = \{v_i \in X : \mbox{deg}_X (v_i)\geq 1\}$. 

We have three cases and in each case, we explain how to get \cfcn{} coloring. 
Cases 1 and 2 use $d+1$ colors and case 3 uses 3 colors.
\begin{enumerate}
    \item There is a clique $K' \subseteq G[V \setminus X]$, with $u\in K'$, such that $|N(u) \cap (X\setminus Y)| \geq 2$. 
    \begin{itemize}
        \item Without loss of generality, let $v_{i_1}, v_{i_2}, \dots , v_{i_m} \in N(u) \cap (X\setminus Y)$, where $m\geq 2$.  
        \item Assign $C(u)=i_1$ and $C(v_{i_\ell})=d+1$, for all $1 \leq \ell \leq m$.
        
        Note that the color $i_2$ is not assigned and will be used for future coloring. 
        \item For each of the uncolored vertices 
        $v_i\in X\setminus Y$, $C(v_i)=i$. 
        \item For each of the cliques $K\subseteq G[V\setminus X]$, 
        \begin{itemize}
            \item If $K$ has a colored vertex, color the remaining vertices with $d+1$.
            \item Else, choose a vertex in $K$ and assign the color $i_2$. 
            Color the remaining vertices with $d+1$. 

        \end{itemize}
    \end{itemize}
        For the vertices $v_{i_\ell}\in N(u)\cap (X\setminus Y)$, where $1\leq \ell \leq m$,
    the vertex $u$ is the uniquely colored neighbor. 
        For all the other vertices $v_i\in X$, $v_i$ itself is the uniquely colored neighbor. 
        For all the vertices in cliques $K \subseteq G[V\setminus X]$, the vertex colored $i_1$ or $i_2$ will serve as the uniquely colored neighbor.

    \item $Y \neq \emptyset$. That is, there exists two vertices $v_i,v_j\in X$ such that $\{v_i,v_j\}\in E(G)$.
    \begin{itemize}
        \item Assign $C(v_i)=i$ and $C(v_j)=d+1$. 
        
        Note that the color $j$ is not used and will be used for 
        future coloring. 
        \item For each of the uncolored vertices $v_k\in X$, assign $C(v_k)=k$. 
        \item For each of the cliques $K$, choose a vertex and 
        assign the color $j$. 
        Color the rest of the vertices with $d+1$. 
    \end{itemize}

    Each vertex in $X\setminus\{v_j\}$ serves as its own uniquely colored neighbor. For the vertex $v_j$, the uniquely colored neighbor is $v_i$. 
    For each clique $K$, the vertex colored $j$ is the uniquely colored neighbor for all the vertices in $K$. 
    
    \item Else, (i) $X$ is an independent set and (ii) for each clique $K$,
    and for all $w\in K$, we have $|N(w)\cap X|\leq 1$. 
    \begin{itemize}
        \item For each clique $K$, choose a vertex and assign the color 1 and color the remaining vertices with 2. 
        \item For all vertices $x\in X$, assign $C(x)=3$. 
    \end{itemize}
    Note that this is a \cfcn{} 3-coloring of $G$. 
    Each vertex in $X$ serves as its own uniquely colored neighbor. For each clique $K$, the vertex colored $1$ acts as a uniquely colored neighbor for all the vertices in $K$.
    
\end{enumerate}
\qed

\section{Proof of Theorem \ref{thm:4col}}\label{sec:outer}
In this section, whenever we refer to an outerplanar
graph $G$, we will also be implicitly referring
to a planar drawing of $G$ with all the vertices appearing in the outer face. We will 
abuse language and say ``faces of $G$'' when 
we want to refer to faces of the above planar 
drawing. 

Theorem \ref{thm:4col} is proved using a two-level induction process. The first level
is using a \emph{block decomposition} of the graph. Any connected graph can be viewed
as a tree of its constituent blocks. We color the blocks in order so that when we
color a block, at most one of its vertices is previously colored. Each block is colored 
without affecting the color of the already colored vertex. The second level 
of the induction is required for coloring each of the blocks. We use
ear decomposition on each block and color the faces of the block in sequence.
However, the proof is quite technical and involves several cases of analysis at each 
step.

We summarize the relevant aspects of block decomposition
below.
The reader is referred to a standard textbook in graph theory \cite{Diestel} 
for more details on this.
\vspace{-0.05in}
\begin{itemize}
    \item A \emph{block} is a maximal connected subgraph without a cut vertex.
    \item Blocks of a connected graph are either maximal 2-connected subgraphs, or edges (the edges which form a block will be bridges).
    \item Two distinct blocks overlap in at most one vertex, which is a cut vertex.
    \item Any connected graph can be viewed as tree of its constituent blocks. 
\end{itemize}
In the following discussion, we explain how to construct a coloring 
$C: V(G) \rightarrow \{1, 2, 3, 4\}$ for an outerplanar graph $G$. At any 
intermediate stage, the coloring
$C$ will satisfy\footnote{The condition marked $\star$ is violated in a few cases. In the 
exceptional cases where it is violated, we shall explain how the cases are handled.} 
the following invariants:
\begin{framed}
\noindent \textbf{Invariants of $C$}
\begin{itemize}
    \item  Every vertex $v$ that has already been assigned a color $C(v)$ has a neighbor $w$, such that $C(w) \neq C(x)$, 
    for all
    $x \in N(v) \setminus \{w\}$. 
    For $v$, the function $U: V(G) \rightarrow \{1, 2, 3, 4\}$ denotes 
    the color of $w$, its uniquely colored neighbor.
    
    \item $\forall v\in V(G)$, $C(v)\neq U(v)$.
    \item $\forall \{v, w\} \in E(G)$, $C(v) \neq C(w)$
     and $|\{C(v),U(v),C(w),U(w)\}|=3$. 
     \textbf{($\star$)}
\end{itemize}
\end{framed}

Theorem \ref{thm:4col} is proved by using an induction
on the block decomposition of the graph $G$ and 
the below results. 

\begin{lemma}\label{lem:5cycles}
If $G$ is a 2-connected outerplanar graph such that all its inner faces contain exactly 5 vertices, then $G$ has a  \cf{} 
coloring using 3 colors.
\end{lemma}

\begin{theorem}\label{thm:otherthan5cycle}
Let $G$ be an outerplanar graph.
\begin{enumerate}
    \item If $B$ is a block of $G$ that is either a bridge, or contains an inner 
face $F$ with $|V(F)|\neq 5$, then $B$ has a  \cf{} 
coloring
using at most 4 colors.

    \item If $B$ is a block of $G$, with exactly one vertex $v$ 
precolored with color $C(v)$ and unique color $U(v)$,
then the rest of $B$ has a \cf{} 
coloring
using at most 4 colors, while retaining $C(v)$ and 
$U(v)$. 
\end{enumerate}
\end{theorem}

\begin{proof}[Proof of Theorem \ref{thm:4col}]
Let $G$ be an outerplanar graph. 
We apply block decomposition on $G$ which results in 
blocks that are either maximal 2-connected subgraphs or single edges. 

If $G$ is 2-connected and all its inner 
faces have exactly 5 vertices, then by Lemma 
\ref{lem:5cycles}, $G$ has a \cf{} coloring using 3 colors.

If $G$ does not fit the above description, 
then $G$ has a block $B$ such that either $B$ is
an edge, or $B$ has an inner face $F$ with 
$|V(F)|\neq 5$. In this case, by Theorem
\ref{thm:otherthan5cycle}.1, $B$ has a  
\cf{} coloring using at most 4 colors.

Viewing $G$ as a tree of its blocks, we can 
start coloring blocks that are adjacent to 
blocks that are already colored. Suppose the
block $B$ is already colored, and let $B'$ be a
block adjacent to $B$. Let $x$ be the cut-vertex between
the blocks $B$ and $B'$.
We use Theorem
\ref{thm:otherthan5cycle}.2 to obtain a 
\cf{} coloring of $B'$ using at most 4 colors, while retaining $C(x)$ and $U(x)$.
\qed
\end{proof}

We now proceed towards proving Lemma \ref{lem:5cycles} and Theorem \ref{thm:otherthan5cycle}.
Lemma \ref{lem:5cycles} and Theorem \ref{thm:otherthan5cycle} 
discusses the coloring of blocks, which is accomplished by means of induction on the faces of the blocks. Towards this end, 
we use the following fact about ear decomposition of 2-connected outerplanar graphs. 
For a proof of the below lemma, we refer the reader to 
\cite{Aubry2016} where this is stated as Observation 2.

\begin{lemma}[Ear Decomposition]\label{lem:ear}
Let $B$ be a 2-connected block in an outerplanar graph. Then $B$ has an ear decomposition $F_0, P_1, P_2, \ldots, P_q$ satisfying the following:
\begin{itemize}
    \item $F_0$ is an arbitrarily chosen inner face of $B$.
    \item Every $P_i$ is a path with end points $v, w$
    such that $\{v, w\}$ is an edge in $F_0 \cup \bigcup_{1 \leq j<i} P_j$.
    Thus $P_i$ together with the edge $\{v, w\}$ forms a face of $B$.
\end{itemize}
\end{lemma}
We are now ready to prove Lemma \ref{lem:5cycles}.
\begin{proof}(Proof of Lemma \ref{lem:5cycles})
Since $G$ is 2-connected, the entire graph forms a single block. 
Let $F_0, P_1,\dots, P_q$ be an ear decomposition of $G$. 
Recall that all the faces have exactly five vertices.
Let $F_0 = v_1-v_2-v_3-v_4-v_5-v_1$.
We assign\footnote{The coloring assigned in this proof
does not satisfy the condition marked $\star$. However,
this is not an issue since we are coloring the whole of
$G$ in this lemma.}
the following colors to the vertices in $F_0$: $C(v_1) = 1, C(v_2) = 1, C(v_3) = 2, C(v_4) = 2, C(v_5) = 3$. 
We also have $U(v_1) = 3, U(v_2) = 2, U(v_3) = 1, U(v_4) = 3, U(v_5) = 1$.

Let $P_i$ be any subsequent face $P_i = w_1-w_2-w_3-w_4-w_5-w_1$ with 
$\{w_1, w_2\}$ being the pre-existing edge in $F_0 \cup \bigcup_{1 \leq j<i} P_j$. Depending on the values already assigned to 
$C(w_1), U(w_1), C(w_2), U(w_2)$, we assign the colors to $w_3, w_4$ and $w_5$. 
We always ensure that $C(v) \neq U(v)$ for all vertices $v$. 
We note that the values $C(w_1), U(w_1), C(w_2), U(w_2)$ can take only 
the four below combinations, w.l.o.g. 
We explain the coloring for the rest of $P_i$ in each of these cases.
\begin{enumerate}
    \item $C(w_1) = C(w_2)$ and $|\{C(w_1), U(w_1), U(w_2)\}| = 3$. W.l.o.g., let $C(w_1) = 1, U(w_1) = 2, C(w_2) = 1, U(w_2) = 3$. Assign $C(w_3) = 2, C(w_4) = 2, C(w_5) = 3$ and $U(w_3) = 1, U(w_4) = 3, U(w_5) = 1$.
    \item $C(w_1) \neq C(w_2)$, $U(w_1) \neq U(w_2)$, and $|\{C(w_1), C(w_2), U(w_1), U(w_2)\}| = 3$. Either $w_1$ serves as the uniquely colored neighbor of $w_2$ or vice versa. W.l.o.g.,
    let $C(w_1) = 1, U(w_1) = 2, C(w_2) = 2, U(w_2) = 3$. Assign $C(w_3) = 1, C(w_4) = 3, C(w_5) = 3$ and $U(w_3) = 2, U(w_4) = 1, U(w_5) = 1$.
    \item $C(w_1) = U(w_2)$ and $C(w_2) = U(w_1)$.
    W.l.o.g.,
    let $C(w_1) = 1, U(w_1) = 2, C(w_2) = 2, U(w_2) = 1$.
    Assign $C(w_3) = 2, C(w_4) = 3, C(w_5) = 1$ and $U(w_3) = 3, U(w_4) = 2, U(w_5) = 3$.
    \item $C(w_1) = C(w_2)$ and $U(w_1) = U(w_2)$.
    W.l.o.g.,
    let $C(w_1) = C(w_2) = 1, U(w_1) = U(w_2) = 2$.
    Assign $C(w_3) = 1, C(w_4) = 2, C(w_5) = 3$ and $U(w_3) = 2, U(w_4) = 3, U(w_5) = 1$.
    \item The case $U(w_1) = U(w_2)$ and $|\{U(w_1), C(w_1), C(w_2)\}| = 3$ does not arise in the above colorings. 
\end{enumerate}
\qed
\end{proof}

At this point, to complete the proof of Theorem \ref{thm:4col}
we need to prove Theorem \ref{thm:otherthan5cycle}. 
We now state 
a few results that would help us towards this end.

\begin{lemma}\label{lem:facecol}
An uncolored face $F$, such that $|V(F)| \neq 5$, can be \cf{} colored using 4 colors satisfying the invariants.
\end{lemma}
\begin{proof}
Let $F=v_1-v_2-v_3-\dots-v_{k-1}-v_k-v_1$ be a  face with $|V(F)|=k$, $k \neq 5$. 
We assign $C(v_1)=1, C(v_2) = 2, C(v_3) = 3$ and for the remaining vertices (if any), we set $C(v_i)=C(v_{i-3})$. In order
to satisfy the invariants, we need to make the following changes:
\begin{itemize}
    \item $k \equiv 0 \pmod 3$. 
    No change is necessary. 
    \item $k \equiv 1 \pmod 3$. Reassign
     $C(v_k) = 4$.
    \item $k \equiv 2 \pmod 3$. 
     Reassign
     $C(v_{k-3}) = 4, C(v_{k-2}) = 2, C(v_{k-1}) = 3, C(v_k) = 4$.
    Notice that this coloring does not satisfy the invariants if $k = 5$. However, the smallest $k$ that we consider in this case is $k = 8$.
\end{itemize}

In each of the above cases the unique color for each vertex $v_i$ is provided by its cyclical successor i.e., $U(v_i) = C(v_{i+1})$.
\qed
\end{proof}

\begin{lemma}~\label{lem:onevertexcolored}
Let $F$ be a face (cycle) in $G$ with one vertex $v$ such that $C(v)$ and $U(v)$ are already assigned, with $C(v) \neq U(v)$.
Then the rest of $F$ can be  \cf{} colored using at most 4 colors, while retaining $C(v)$ and $U(v)$, and 
satisfying the invariants.
\end{lemma}

\begin{proof}
Let $v_1$ be the colored vertex in 
the cycle $F$. 
We may assume w.l.o.g. that 
$C(v_1)=1$ and $U(v_1)=2$. Now, we extend $C$ to the remainder of $F$. 
\begin{itemize}
    \item $|V(F)|=3$ with $F=v_1-v_2-v_3-v_1$. 

    We assign: $C(v_2)=3$, $C(v_3)=4$ and $U(v_2)=1$, $U(v_3)=1$.
    
    
    
    \item $|V(F)|\geq 4$ with $F=v_1-v_2-v_3-\dots-v_{k-1}-v_k-v_1$. 
    
    We first assign: $C(v_2)=3$ and $C(v_{3})=2$. 
    For the remaining vertices $v_i$, we set $C(v_i) = C(v_{i-3})$ 
    for $4 \leq i \leq k$. However, we need to make some changes to this in order 
    to satisfy the invariants. We have the following subcases:
    \begin{itemize}
        \item $k \equiv 0 \mbox{ or } 1 \pmod 3$. 
        Reassign $C(v_k)=4$.
        \item $k \equiv 2 \pmod 3$. 
        Reassign $C(v_{k-1})=4$.
    \end{itemize}
    In each of the above cases the unique color for each vertex $v_i$ is provided by its cyclical successor i.e., $U(v_i) = C(v_{i+1})$. Observe that $U(v_1)$ is left unchanged, by ensuring $v_2$ and $v_k$, the neighbors of
    $v_1$, are not assigned the color $U(v_1)$.
\end{itemize}
\qed
\end{proof}

\begin{lemma}\label{lem:CuCvsame}
Let $F$ be a face with $|V(F)|\geq 4$ with such that the edge $\{v_1, v_2\} \in E(F)$ 
and $v_1$ and $v_2$ already colored such that 
$C(v_1)=C(v_2)$ and $U(v_1)\neq U(v_2)$.
Then the rest of $F$ can be \cf{} colored using 4 colors satisfying the invariants. 
\end{lemma}

\begin{proof}
W.l.o.g., we may assume $C(v_1)=C(v_2)=4$, $U(v_1)=1$ and $U(v_2)=2$. We have the following
cases:
\begin{itemize}
    \item $|V(F)|=4$ with $F=v_2-v_3-v_4-v_1-v_2$. 
    We assign: $C(v_3)=1$, $C(v_4)=3$ and $U(v_3)=4$ and $U(v_4)=4$.
     \item If $|V(F)|= 5$ with $F=v_2-v_3-v_4-v_5-v_1-v_2$.
     We assign: $C(v_3)=1$, $C(v_4)=2$, $C(v_5)=3$ and $U(v_3)=2$, $U(v_4)=3$ and $U(v_5)=4$.
    \item If $|V(F)|\geq 6$ with $F=v_2-v_3-\dots- v_{k-1} - v_{k}-v_1 -v_2$. 
    We assign: $C(v_3)=3$ and $C(v_4)=2$. For all $5\leq i \leq k$, $C(v_i)=C(v_{i-3})$.
    \begin{itemize}
        \item $k \equiv 0 \pmod 3$. Reassign $C(v_{k-1})=1$. 
        \item $k \equiv 1 \pmod 3$. No change is required.
         \item $k \equiv 2 \pmod 3$. Reassign $C(v_{k-1})=1$ and $C(v_k)=2$.
    \end{itemize}
    The unique color of each vertex $v_i$ is provided its cyclical successor i.e., $U(v_i) = C(v_{i+1})$.
\end{itemize}
\qed
\end{proof}

\begin{lemma}~\label{lem:path}
Let $P$ be a path in $G$ whose endpoints are $v_1, v_2$. Suppose
$\{v_1, v_2\} \in E(G)$ and that $v_1, v_2$ are already assigned 
the functions $C$ and $U$ satisfying the invariants.
Then the rest of $P$ can be \cf{} colored using at most 4 colors, while retaining $C$ and $U$ values of the endpoints, and satisfying the invariants.
\end{lemma}
Since the proof of the above lemma is a bit long and involved, we first prove Theorem \ref{thm:otherthan5cycle} using Lemmas \ref{lem:facecol}, \ref{lem:onevertexcolored} and \ref{lem:path}.

\begin{proof}[Proof of Theorem \ref{thm:otherthan5cycle}]
\begin{enumerate}
    \item If the block is a bridge, say $\{v, w\}$, then we color it
    $C(v) = 1, C(w) = 2$ with $U(v) = 2, U(w) = 1$. Note that the 
    invariant marked $\star$ is violated in this case. However, this does
    not cause an issue since this edge is a bridge, and it does not appear 
    in any inner face.
    
    \vspace{0.05in}
    If the block is not a bridge, then by assumption, it contains a face $F$ 
    such that $|V(F)| \neq 5$. By Lemma \ref{lem:facecol}, we have a coloring
    of $F$ using 4 colors and satisfying the invariants. By the Lemma \ref{lem:ear} (Ear Decomposition),
    the block has an ear decomposition $F, P_1, P_2, \ldots$ with $F$ as the 
    starting inner face. Recall that for every path $P_i$, the end points form
    an edge in $F_0 \cup \bigcup_{1 \leq j<i} P_j$. 
    We color the paths $P_1, P_2, \ldots$ in this order. 
    By Lemma \ref{lem:path}, we have a coloring for each of these paths 
    using 4 colors and satisfying the invariants. 
    \item Let $v$ be the vertex in the block that is already colored. 
    W.l.o.g., we may assume that 
    $C(v) = 1$ and $U(v) = 2$.
    
    \vspace{0.05in}
    If the block is a bridge $\{v, w\}$,  we color $w$ with $C(w)=3$ and set $U(w) = 1$.
    
    \vspace{0.05in}
    If the block is not a bridge, choose an inner face $F$ that contains $v$.
    Using Lemma \ref{lem:onevertexcolored}, we color the remainder of $F$ using 
    at most 4 colors and satisfying the invariants. The rest of the proof 
    follows from the fact that we have an ear decomposition with $F$ as the 
    starting face, and Lemma \ref{lem:path}.
    This is very similar to the argument in the proof of part 1 of this theorem
    and hence the details are omitted.
\end{enumerate}
\qed
\end{proof}



\noindent In order to complete the proof of Theorem \ref{thm:4col}, the last remaining 
piece is the proof of Lemma \ref{lem:path}. 

\begin{proof}[Proof of Lemma~\ref{lem:path}] Let $v_1$ and $v_2$ be the end points of $P$. We extend the coloring $C$ to the remainder of $P$. 
According to the invariants of $C$, we have only 2 cases possible.

\noindent\textbf{Case 1:} $C(v_1) \neq C(v_2), U(v_1) \neq U(v_2)$. W.l.o.g. we may assume  
$C(v_1)=1$, $C(v_2)=2$ and $U(v_1)=2$, $U(v_2)=3$.
\begin{itemize}
    \item  $|V(P)|=3$, $P=v_2-v_3-v_1$. Assign $C(v_3)=4$ with $U(v_3)=2$.
    \item  $|V(P)|=4$, $P=v_2-v_3-v_4-v_1$. Assign $C(v_3)=4$, $C(v_4)=3$ with $U(v_3)=3$, $U(v_4)=1$. 

    \item  $|V(P)|\geq5$, $P=v_2-v_3-\dots -v_{k-1}-v_k-v_1$. 
    We first assign $C(v_3)=1, C(v_4) = 3, C(v_5) =4$. 
    For the remaining vertices $v_i$, we initially assign $C(v_i) = C(v_{i-3})$ 
    for $6 \leq i \leq k$. However, we need to make some changes 
    to satisfy the invariants. We have the following subcases:
    \begin{itemize}
        \item $k \equiv 0 \pmod 3$. 
        Reassign $C(v_{k-1})=2$ and $C(v_k) = 4$
        \item $k \equiv 1 \pmod 3$. 
        Reassign $C(v_{k-1})=2$.
        \item $k \equiv 2 \pmod 3$. 
        No change is necessary.
    \end{itemize}
    In each of the above cases the unique color for each vertex $v_i$ is provided by its cyclical successor i.e., $U(v_i) = C(v_{i+1})$.
\end{itemize}

\noindent \textbf{Case 2:} $U(v_1) = U(v_2)$. W.l.o.g., we may assume $C(v_1)=1$, $C(v_2)=2$ and $U(v_1)= U(v_2)=3$. 

\begin{itemize}
    \item  \textbf{Case 2(i):} $|V(P)|=3$ and $P=v_2-v_3-v_1$.
    \begin{itemize}
        \item \textbf{Case 2(i)(a):} Vertices $v_1$ and $v_2$ are the only
        neighbors of $v_3$. Assign $C(v_3)=4$ and $U(v_3)=2$. The  
        invariant marked $\star$ is not satisfied, but that does not matter as 
        $v_3$ does not participate in any further faces.
        
        \item \textbf{Case 2(i)(b):} One of the edges $\{v_1, v_3\}$ or $\{v_2, v_3\}$ does
        not feature in an another face. W.l.o.g., say $\{v_2, v_3\}$ be that edge. 
        Assign $C(v_3) = 4$ with $U(v_3) = 1$. The $\star$ invariant is violated
        for $\{v_2, v_3\}$ here but it does not affect the further coloring.
        
        \item \textbf{Case 2(i)(c):} One of the edges $\{v_1, v_3\}$ or $\{v_2, v_3\}$ 
         features in an uncolored face $F$ such that $|V(F)| \neq 3$. 
         W.l.o.g., say $\{v_2, v_3\}$ is that edge.

         We assign $C(v_3) = 4$ with $U(v_3) = 1$. Let $|V(F)| = k$ with 
         $F=v_3-w_1-w_2-\dots-w_{k-2}-v_2-v_3$.
         We assign $C(w_1)=3$, $C(w_2)=1$ and $C(w_3)=4$ (if $w_3$ exists). 
         For all $4\leq i\leq k-2$, $C(w_i)=C(w_{i-3})$. 
         If $k \equiv 0 \pmod 3$, we reassign $C(w_{k-4})=2$, $C(w_{k-3}) = 1$ and $C(w_{k-2})=4$.
         
         The unique colors $U$ for the vertices
         are assigned as follows:
         \begin{itemize}
             \item For $k=6$, $U(w_1) = 4, U(w_2) = 3, U(w_3) = 2$ and
             $U(w_4) = 2$.
             \item For $k\neq 6$, we have 
             for $1\leq i \leq k-3$, 
         $U(w_i) = C(w_{i+1})$ and $U(w_{k-2})  = C(v_2) = 2$. 
         \end{itemize}


    
    
    

         
\begin{figure}[t!]
\centering
\begin{tikzpicture}[every node/.style={node distance=1.8cm,scale=0.9}, scale = 0.8]
\tikzstyle{vertex}=[circle,draw, minimum size=8pt]
\tikzstyle{edge} = [draw,thick,-,black]
\node[vertex]  (v3) at (0,-2) {$v_3$};
\node[vertex]  (v1) at (-1.2,2) {$v_1$};
\node[vertex]  (v2) at (1.2,2) {$v_2$};
\node[vertex]  (x) at (-3,0) {$x$};
\node[vertex]  (y) at (3,0) {$y$};
\node[vertex]  (z) at (-2.2,-2.5) {$z$};

\draw[edge, color=black] (x) -- (v1) -- (v2) -- (y);
\draw[edge, color=black] (x) -- (v3) -- (v1);
\draw[edge, color=black] (y) -- (v3) -- (v2);
\draw[edge, color=black] (x) -- (z) -- (v3);
\end{tikzpicture}
\caption{Case 2(i)(d)}\label{fig:2id}
\end{figure}
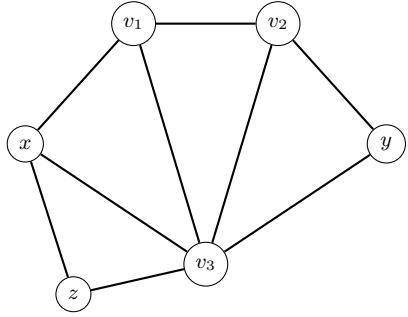

        \item \textbf{Case 2(i)(d):} The only remaining case is when both
         the edges $\{v_1, v_3\}$ or $\{v_2, v_3\}$ feature in uncolored triangular
         faces. Let $\{v_1, v_3\}$ form a triangular face with $x$ and 
         $\{v_2, v_3\}$ with $y$. We have two subcases:
        \begin{itemize}
            \item The edge $\{x, v_3\}$ forms a triangular face with another 
            vertex $z$ (see Figure \ref{fig:2id}).  
            Assign $C(v_3) = 1, C(x) = 2, C(y) = 4, C(z) = 3$ and $U(v_3) = 4, U(x) = 3, 
            U(y) = 2, U(z) = 1$. Some edges 
            violate the invariant marked $\star$, but these edges are already
            part of two faces, and hence do not feature in the further coloring.
            
            \item The edge $\{x, v_3\}$ is not part of a triangular face with another
            vertex. In this case, we assign $C(v_3) = 4, C(x) = 4, C(y) = 1$ and $U(v_3) = 2, U(x) = 1, U(y) = 2$. Out of the edges that violate 
            the invariant marked $\star$, the only one that can participate in the
            further coloring is the edge $\{x, v_3\}$. By assumption, $\{x, v_3\}$
            is not part of a triangular face. In Lemma \ref{lem:CuCvsame}, 
            we explain how to color the uncolored face that is $\{x, v_3\}$ may be 
            a part of.
        \end{itemize}
    \end{itemize}
    
    \item \textbf{Case 2(ii):} $|V(P)|=4$, $P=v_2-v_3-v_4-v_1$. 
    \begin{itemize}
        \item \textbf{Case 2(ii)(a):} The edge $\{v_3, v_4\}$ forms a triangular
        face with a vertex $x$.
        We assign $C(v_3)= 1, C(v_4) = 4, C(x) = 3$, with $U(v_3)= 3, U(v_4) = 1, U(x) = 4$.
        \item \textbf{Case 2(ii)(b):} The edge $\{v_3, v_4\}$ is not part of 
        an uncolored triangular face. We assign $C(v_3) = C(v_4) = 4$, with 
        $U(v_3)= 2, U(v_4) = 1$. If the edge  $\{v_3, v_4\}$ is part of an
        uncolored face $F$, by assumption, we know that $|V(F)| \geq 4$ and hence
        we can use Lemma \ref{lem:CuCvsame} to color $F$ satisfying the invariants.
    \end{itemize}
    \item \textbf{Case 2(iii):} $|V(P)|= 5$ with $P = v_2-v_3-v_4-v_5-v_1$. We
    assign $C(v_3)=1, C(v_4) = 3, C(v_5) =2$, with $U(v_3)=3, U(v_4) = 2, U(v_5) =1$.
    \item \textbf{Case 2(iv):} $|V(P)|\geq 6$, with $P=v_2-v_3-\dots-v_{k-2}-v_{k-1}-v_k-v_1$.

    We first assign $C(v_3)=4$ and $C(v_4)=3$. 
    For $5\leq i \leq k$, assign $C(v_i)=C(v_{i-3})$. 
    If $k \equiv 1 \pmod 3$, then reassign $C(v_{k-2})=1$ and $C(v_k)=2$. 
    For each vertex $v_i$, the unique color is provided by its cyclical successor
    i.e., $U(v_i) = C(v_{i+1})$.
\end{itemize} 
\qed
\end{proof}

\noindent\textbf{Algorithmic Note:} The steps in the proof of Theorem \ref{thm:4col}
leads to an algorithm. Block decomposition, outerplanarity testing and embedding
outerplanar graphs \cite{outerplanar} can all be done in linear time, i.e., $O(|V(G)|)$. Thus we have 
an $O(|V(G)|)$ time algorithm, that given an outerplanar graph $G$, determines a 
\cf{} coloring for $G$ that uses four colors.

\end{document}